\documentclass[a4paper, intlimits, reqno]{amsart}

\usepackage[english]{babel}
\usepackage[T1]{fontenc}
\usepackage[utf8]{inputenc}

\usepackage{amsmath}
\usepackage{amssymb}
\usepackage{MnSymbol}
\usepackage{amsthm}
\allowdisplaybreaks
\usepackage{amsfonts}
\usepackage{mathrsfs} 
\usepackage{mathtools}
\usepackage{xfrac}
\usepackage{faktor}
\usepackage{enumitem}
\usepackage{multicol}
\usepackage{url}
\usepackage{dsfont}
\usepackage[numbers]{natbib}
\usepackage{prettyref}

\newrefformat{def}{Definition \ref{#1}}
\newrefformat{rem}{Remark \ref{#1}}
\newrefformat{sect}{Section \ref{#1}}
\newrefformat{sub}{Section \ref{#1}}
\newrefformat{prop}{Proposition \ref{#1}}
\newrefformat{thm}{Theorem \ref{#1}}
\newrefformat{lem}{Lemma \ref{#1}}
\newrefformat{chap}{Chapter \ref{#1}}
\newrefformat{cor}{Corollary \ref{#1}}
\newrefformat{par}{Paragraph \ref{#1}}
\newrefformat{ex}{Example \ref{#1}}
\newrefformat{part}{Part \ref{#1}}
\newrefformat{fig}{Figure \ref{#1}}
\newrefformat{app}{Appendix \ref{#1}}
\newrefformat{que}{Question \ref{#1}}
\newrefformat{cond}{Condition \ref{#1}}
\newrefformat{conv}{Convention \ref{#1}}

\swapnumbers
\newtheoremstyle{dotless}{}{}{\itshape}{}{\bfseries}{}{}{}
\theoremstyle{dotless}
\theoremstyle{plain}
\newtheorem{thm}{Theorem}[section]
\newtheorem{lem}[thm]{Lemma}
\newtheorem{prop}[thm]{Proposition}
\newtheorem{cor}[thm]{Corollary}
\theoremstyle{definition}
\newtheorem{defn}[thm]{Definition}
\newtheorem{rem}[thm]{Remark}

\newcommand{\N} {\mathbb{N}}
\newcommand{\Z} {\mathbb{Z}}

\newcommand{\R} {\mathbb{R}}
\newcommand{\C} {\mathbb{C}}
\newcommand{\K} {\mathbb{K}}

\newcommand{\D} {\mathbb{D}}

\newcommand{\acx} {\operatorname{acx}}
\newcommand{\oacx} {\overline{\operatorname{acx}}}

\DeclareMathOperator{\id}{id}

\DeclareMathOperator{\Span}{span}
\providecommand{\differential}{\mathrm{d}}
\renewcommand{\d}{\differential}

\begin{document}

\title[Series representations]{Series representations in spaces of vector-valued functions via Schauder decompositions}
\author[K.~Kruse]{Karsten Kruse}
\address{TU Hamburg \\ Institut f\"ur Mathematik \\
Am Schwarzenberg-Campus~3 \\
Geb\"aude E \\
21073 Hamburg \\
Germany}
\email{karsten.kruse@tuhh.de}



\date{\today}
\begin{abstract}
It is a classical result that every $\C$-valued holomorphic function has a local power series representation. 
This even remains true for holomorphic functions with values in a locally complete locally convex Hausdorff space $E$ over 
$\C$. Motivated by this example we try to answer the following question. Let $E$ be a locally convex Hausdorff space 
over a field $\K$, $\mathcal{F}(\Omega)$ be a locally convex Hausdorff space of $\K$-valued functions on a set $\Omega$
and $\mathcal{F}(\Omega,E)$ be an $E$-valued counterpart of $\mathcal{F}(\Omega)$ 
(where the term $E$-valued counterpart needs clarification itself). 
For which spaces is it possible to lift series representations of elements of $\mathcal{F}(\Omega)$ to elements 
of $\mathcal{F}(\Omega,E)$?  
We derive sufficient conditions for the answer to be affirmative using Schauder decompositions 
which are applicable for many classical spaces of functions 
$\mathcal{F}(\Omega)$ having an equicontinuous Schauder basis. 
\end{abstract}
\maketitle
\section{Introduction}
The purpose of this paper is to lift series representations known from scalar-valued functions to vector-valued functions 
and its underlying idea was derived from the classical example of the (local) power series representation of a holomorphic function. 
Let $\D_{r}\subset\C$ be an open disc around zero with radius $r>0$ and $\mathcal{O}(\D_{r})$ 
be the space of holomorphic functions on $\D_{r}$, i.e.\ the space of functions $f\colon\D_{r}\to \C$ such that 
the limit
\begin{equation}\label{intro:holom}
 f^{(1)}(z):=\lim_{\substack{h\to 0\\ h\in\C,h\neq 0}}\frac{f(z+h)-f(z)}{h},\quad z\in \D_{r},
\end{equation}
exists in $\C$. It is well-known that every $f\in\mathcal{O}(\D_{r})$ can be written as 
\[
 f(z)=\sum_{n=0}^{\infty}\frac{f^{(n)}(0)}{n!}z^{n},\quad z\in\D_{r},
\]
where the power series on the right-hand side converges uniformly on every compact subset of $\D_{r}$ and 
$f^{(n)}(0)$ is the $n$-th complex derivative of $f$ at $0$ which is defined from \eqref{intro:holom} by the recursion $f^{(0)}:=f$ and 
$f^{(n)}:=(f^{(n-1)})^{(1)}$ for $n\in\N$. Amazingly by \cite[2.1 Theorem and Definition, p.\ 17-18]{grosse-erdmann1992} and 
\cite[5.2 Theorem, p.\ 35]{grosse-erdmann1992}, this series representation remains valid if 
$f$ is a holomorphic function on $\D_{r}$ with values in a locally complete locally convex Hausdorff space $E$ over $\C$ 
where holomorphy means that the limit \eqref{intro:holom} exists in $E$ and the higher complex derivatives are defined 
recursively as well. 
Analysing this example, we observe that $\mathcal{O}(\D_{r})$, equipped with the topology of uniform convergence on 
compact subsets of $\D_{r}$, is a Fr\'{e}chet space, in particular barrelled, with a Schauder basis formed 
by the monomials $z\mapsto z^{n}$. Further, the formulas for the complex derivatives of a 
$\C$-valued resp.\ an $E$-valued function $f$ on $\D_{r}$ 
are built up in the same way by \eqref{intro:holom}.

Our goal is to derive a mechanism which uses these observations and transfers known series representations for 
other spaces of scalar-valued functions to their vector-valued counterparts. Let us describe the general setting. 
We recall from \cite[14.2, p.\ 292]{Jarchow} that a sequence $(f_{n})$ in a locally convex Hausdorff space $F$ over a 
field $\K$ is called a topological basis, 
or simply a basis, if for every $f\in F$ there is a unique sequence of coefficients 
$(\lambda^{\K}_{n}(f))$ in $\K$ such that 
\begin{equation}\label{intro:basis}
f=\sum_{n=1}^{\infty}\lambda^{\K}_{n}(f)f_{n}
\end{equation}
where the series converges in $F$. Due to the uniqueness of the coefficients the map 
$\lambda^{\K}_{n}\colon f\mapsto \lambda^{\K}_{n}(f)$ is well-defined, linear and 
called the $n$-th coefficient functional associated to $(f_{n})$. Further, for each $k\in\N$ the 
map $P^{\K}_{k}\colon F\to F$, $P^{\K}_{k}(f):=\sum_{n=1}^{k}\lambda^{\K}_{n}(f)f_{n}$, 
is a linear projection whose range is $\operatorname{span}(f_{1},\ldots,f_{n})$ and it is called 
the $k$-th expansion operator associated to $(f_{n})$. 
A basis $(f_{n})$ of $F$ is called equicontinuous if the expansion operators $P^{\K}_{k}$ 
form an equicontinuous sequence in the linear space $L(F,F)$ of continuous linear maps from $F$ to $F$ 
(see \cite[14.3, p.\ 296]{Jarchow}). 
A basis $(f_{n})$ of $F$ is called a Schauder basis if the coefficient functionals are continuous 
which, in particular, is already fulfilled if $F$ is a Fr\'{e}chet space by 
\cite[Corollary 28.11, p.\ 351]{meisevogt1997}. If $F$ is barrelled, then 
a Schauder basis of $F$ is already equicontinuous and $F$ has the (bounded) approximation property 
by the uniform boundedness principle.

The starting point for our approach is equation \eqref{intro:basis} . 
Let $F$ and $E$ be locally convex Hausdorff spaces over a field $\K$ where $F$ has an equicontinuous Schauder basis $(f_{n})$ 
with associated coefficient functionals $(\lambda^{\K}_{n})$.
The expansion operators $(P^{\K}_{k})$ form a so-called Schauder decomposition of $F$ (see \cite[p.\ 77]{Bonet2007}), i.e.\ 
they are continuous projections on $F$ such that
\begin{enumerate}
 \item [(i)] $P^{\K}_{k}P^{\K}_{j}=P^{\K}_{\min(j,k)}$ for all $j,k\in\N$,
 \item [(ii)] $P^{\K}_{k}\neq P^{\K}_{j}$ for $k\neq j$,
 \item [(iii)] $(P^{\K}_{k}f)$ converges to $f$ for each $f\in F$.
\end{enumerate}
This operator theoretic definition of a Schauder decomposition is equivalent to the usual definition in terms of closed subspaces of $F$ 
given in \cite[p.\ 377]{kalton_1970} (see \cite[p.\ 219]{Lotz1985}).
In our main \prettyref{thm:schauder_decomp} we prove that $(P^{\K}_{k}\varepsilon\id_{E})$ is a Schauder decomposition 
of the Schwartz' $\varepsilon$-product $F\varepsilon E :=L_{e}(F_{\kappa}',E)$ 
and each $u\in F\varepsilon E$ has the series representation
\[
u(f')=\sum_{n=1}^{\infty}u(\lambda_{n}^{\K})f'(f_{n}),\quad f'\in F'.
\]
Now, suppose that $F=\mathcal{F}(\Omega)$ is a space of $\K$-valued functions on a set $\Omega$ with a 
topology such that the point-evaluation functionals $\delta_{x}$, $x\in\Omega$, are continuous 
and that there is a locally convex Hausdorff space $\mathcal{F}(\Omega,E)$ of 
functions from $\Omega$ to $E$ such that the map 
\[
S\colon \mathcal{F}(\Omega)\varepsilon E\to \mathcal{F}(\Omega,E),\; u\longmapsto [x\mapsto u(\delta_{x})],
\]  
is a (linear topological) isomorphism.
Assuming that for each $n\in\N$ and $u\in\mathcal{F}(\Omega)\varepsilon E$ there is 
$\lambda^{E}_{n}(S(u))\in E$ with
\begin{equation}\label{eq:intro_consistent}
\lambda^{E}_{n}(S(u))=u(\lambda^{\K}_{n}),
\end{equation}
we obtain in \prettyref{cor:schauder_decomp} that $(S\circ (P_{k}^{\K}\varepsilon\id_{E})\circ S^{-1})_{k}$ is a Schauder decomposition of $\mathcal{F}(\Omega,E)$ and
\[
f=\lim_{k\to\infty}(S\circ (P_{k}^{\K}\varepsilon\id_{E})\circ S^{-1})_{k}(f)=\sum_{n=1}^{\infty}\lambda^{E}_{n}(f)f_{n},\quad f\in\mathcal{F}(\Omega,E),
\]
which is the desired series representation in $\mathcal{F}(\Omega,E)$. 
Condition \eqref{eq:intro_consistent} might seem strange at a first glance 
but for example in the case of $E$-valued holomorphic functions 
on $\D_{r}$ it guarantees that the complex derivatives at $0$ appear in the Schauder decomposition of $\mathcal{O}(\D_{r},E)$ 
since $S(u)^{(n)}(0)=u(\delta_{0}^{(n)})$ for all $u\in\mathcal{O}(\D_{r})\varepsilon E$ and $n\in\N_{0}$ 
where $\delta_{0}^{(n)}$ is the point-evaluation of the $n$-th complex derivative.
We apply our result to sequence spaces, spaces of continuously differentiable functions on 
a compact interval, the space of holomorphic functions, the Schwartz space and the space of 
smooth functions which are $2\pi$-periodic in each variable. 

As a byproduct of \prettyref{thm:schauder_decomp} we obtain  
that every element of the completion $F\,\widehat{\otimes}_{\varepsilon}E$ of the injective tensor product 
has a series representation as well if $F$ is a complete space with an equicontinuous Schauder basis and $E$ is complete. 
Concerning series representation in $F\,\widehat{\otimes}_{\varepsilon}E$, little seems to be known
whereas for the completion $F\,\widehat{\otimes}_{\pi}E$ of the projective tensor product $F\otimes_{\pi}E$ 
of two metrisable locally convex spaces $F$ and $E$ it is well-known that every $f\in F\,\widehat{\otimes}_{\pi}E$ 
has a series representation 
\[
f=\sum_{n=1}^{\infty}a_{n}f_{n}\otimes e_{n}
\]
where $(a_{n})\in \ell^{1}$, i.e.\ $(a_{n})$ is absolutely summable, 
and $(f_{n})$ and $(e_{n})$ are null sequences in $F$ and $E$, respectively 
(see e.g.\ \cite[Chap.\ I, \S2 , n$^{\circ}$1, Th\'{e}or\`{e}me 1, p.\ 51]{Gro} 
or \cite[15.6.4 Corollary, p.\ 334]{Jarchow}). 
If $F$ and $E$ are metrisable and one of them is nuclear, then the isomorphy
$F\,\widehat{\otimes}_{\pi}E \cong F\,\widehat{\otimes}_{\varepsilon}E$ holds and we trivially have a series 
representation of the elements of $F\,\widehat{\otimes}_{\varepsilon}E$ as well.
Other conditions on the existence of series representations of the elements of $F\,\widehat{\otimes}_{\varepsilon}E$
can be found in \cite[Proposition 4.25, p.\ 88]{ryan2002}, where $F$ and $E$ are Banach spaces 
and both of them have a Schauder basis, and in \cite[Theorem 2, p.\ 283]{Joiner1970}, where $F$ and $E$ 
are locally convex Hausdorff spaces and both of them have an equicontinuous Schauder basis.
\section{Notation and Preliminaries}
We equip the spaces $\R^{d}$, $d\in\N$, and $\C$ with the usual Euclidean norm $|\cdot|$. 
For a subset $M$ of a topological vector space $X$, we write $\oacx(M)$ 
for the closure of the absolutely convex hull $\acx(M)$ of $M$ in $X$.

By $E$ we always denote a non-trivial locally convex Hausdorff space, in short lcHs, over the field 
$\K=\R$ or $\C$ equipped with a directed fundamental system of 
seminorms $(p_{\alpha})_{\alpha\in \mathfrak{A}}$.
If $E=\K$, then we set $(p_{\alpha})_{\alpha\in \mathfrak{A}}:=\{|\cdot|\}.$ 
We recall that for a disk $D\subset E$, i.e.\ a bounded, absolutely convex set, 
the vector space $E_{D}:=\bigcup_{n\in\N}nD$ becomes a normed space if it is equipped with the
gauge functional of $D$ as a norm (see \cite[p.\ 151]{Jarchow}). The space $E$ is called locally 
complete if $E_{D}$ is a Banach space for every closed disk $D\subset E$ (see \cite[10.2.1 Proposition, p.\ 197]{Jarchow}).
For more details on the theory of locally convex spaces see \cite{F/W/Buch}, \cite{Jarchow} or \cite{meisevogt1997}.

By $X^{\Omega}$ we denote the set of maps from a non-empty set $\Omega$ to a non-empty set $X$, 
by $\chi_{K}$ the characteristic function of a subset $K\subset\Omega$ and by $L(F,E)$ 
the space of continuous linear operators from $F$ to $E$ where $F$ and $E$ are locally convex Hausdorff spaces. 
If $E=\K$, we just write $F':=L(F,\K)$ for the dual space.
If $F$ and $E$ are (linearly topologically) isomorphic, we write $F\cong E$ and, if $F$ is only isomorphic to a subspace of $E$, 
we write $F\tilde{\hookrightarrow} E$.
We denote by $L_{t}(F,E)$ the 
space $L(F,E)$ equipped with the locally convex topology of uniform convergence 
on the absolutely convex compact subsets of $F$ if $t=\kappa$ and
on the precompact (totally bounded) subsets of $F$ if $t=\gamma$.

The so-called $\varepsilon$-product of Schwartz is defined by 
\begin{equation}\label{notation0}
F\varepsilon E:=L_{e}(F_{\kappa}',E)
\end{equation}
where $L(F_{\kappa}',E)$ is equipped with the topology of uniform convergence on equicontinuous subsets of $F'$. 
This definition of the $\varepsilon$-product coincides with the original 
one by Schwartz \cite[Chap.\ I, \S1, D\'{e}finition, p.\ 18]{Sch1}. 
It is symmetric which means that $F\varepsilon E\cong E\varepsilon F$. In the literature the definition of the 
$\varepsilon$-product is sometimes done the other way around, i.e.\ $E\varepsilon F$ is defined by the right-hand side 
of \eqref{notation0} but due to the symmetry these definitions are equivalent and for our purpose the given definition 
is more suitable. If we replace $F_{\kappa}'$ by $F_{\gamma}'$, we obtain Grothendieck's definition of the 
$\varepsilon$-product and we remark that the two $\varepsilon$-products coincide 
if $F$ is quasi-complete because then $F_{\gamma}'=F_{\kappa}'$. Jarchow uses a third, different definition 
of the $\varepsilon$-product (see \cite[16.1, p.\ 344]{Jarchow}) which coincides with the one of Schwartz 
if $F$ is complete by \cite[9.3.7 Proposition, p.\ 179]{Jarchow}. However, we stick to Schwartz' definition.

For locally convex Hausdorff spaces $F_{i}$, $E_{i}$ and $T_{i}\in L(F_{i},E_{i})$, $i=1,2$, we define the $\varepsilon$-product 
$T_{1}\varepsilon T_{2}\in L(F_{1}\varepsilon F_{2},E_{1}\varepsilon E_{2})$ of the operators $T_{1}$ and $T_{2}$ by 
\[
 (T_{1}\varepsilon T_{2})(u):=T_{2}\circ u\circ T_{1}^{t},\quad u\in F_{1}\varepsilon F_{2},
\]
where $T_{1}^{t}\colon E_{1}'\to F_{1}'$, $e'\mapsto e'\circ T_{1}$, is the dual map of $T_{1}$. 
If $T_{1}$ is an isomorphism and $F_{2}=E_{2}$, then $T_{1}\varepsilon \id_{E_{2}}$ 
is also an isomorphism with inverse $T_{1}^{-1}\varepsilon\id_{E_{2}}$ by \cite[Chap.\ I, \S1, Proposition 1, p.\ 20]{Sch1} 
(or \cite[16.2.1 Proposition, p.\ 347]{Jarchow} if the $F_{i}$ are complete).

As usual we consider the tensor product $F\otimes E$ as an algebraic subspace of $F\varepsilon E$ for two locally convex Hausdorff 
spaces $F$ and $E$ by means of the linear injection 
\[
\Theta\colon F\otimes E\to F\varepsilon E,\; 
\sum^{k}_{n=1}{f_{n}\otimes e_{n}}\longmapsto\bigl[y\mapsto \sum^{k}_{n=1}{y(f_{n}) e_{n}}\bigr].
\]
Via $\Theta$ the space $F\otimes E$ is identified with space of operators with finite rank in $F\varepsilon E$ 
and a locally convex topology is induced on $F\otimes E$. 
We write $F\otimes_{\varepsilon} E$ for $F\otimes E$ equipped 
with this topology and $F\,\widehat{\otimes}_{\varepsilon} E$ for the completion 
of the injective tensor product $F\otimes_{\varepsilon}E$.
By $\mathfrak{F}(E)$ we denote the space of linear operators from $E$ to $E$ with finite rank. 
A locally convex Hausdorff space $E$ is said to have (Schwartz') approximation property (AP) 
if the identity $\id_{E}$ on $E$ is contained in the closure of $\mathfrak{F}(E)$ in $L_{\kappa}(E,E)$. 
The space $E$ has AP if and only if $E\otimes F$ is dense in $E\varepsilon F$ 
for every locally convex Hausdorff space (every Banach space) $F$ by \cite[Satz 10.17, p.\ 250]{Kaballo}.
The space $E$ has Grothendieck's approximation property if $\id_{E}$ is contained 
in the closure of $\mathfrak{F}(E)$ in $L_{\gamma}(E,E)$. If $E$ is quasi-complete, 
both approximation properties coincide. 
For more information on the theory of $\varepsilon$-products 
and tensor products see \cite{Defant}, \cite{Jarchow} and \cite{Kaballo}. 

A function $f\colon\Omega\to E$ on an open set $\Omega\subset\R^{d}$ to an lcHs $E$ is called 
continuously partially differentiable ($f$ is $\mathcal{C}^{1}$) 
if for the $n$-th unit vector $e_{n}\in\R^{d}$ the limit
\[
(\partial^{e_{n}})^{E}f(x):=\lim_{\substack{h\to 0\\ h\in\R,h\neq 0}}\frac{f(x+he_{n})-f(x)}{h}
\]
exists in $E$ for every $x\in\Omega$ and $(\partial^{e_{n}})^{E}f$ 
is continuous on $\Omega$ ($(\partial^{e_{n}})^{E}f$ is $\mathcal{C}^{0}$) for every $1\leq n\leq d$.
For $k\in\N$ a function $f$ is said to be $k$-times continuously partially differentiable 
($f$ is $\mathcal{C}^{k}$) if $f$ is $\mathcal{C}^{1}$ and all its first partial derivatives are $\mathcal{C}^{k-1}$.
A function $f$ is called infinitely continuously partially differentiable ($f$ is $\mathcal{C}^{\infty}$) 
if $f$ is $\mathcal{C}^{k}$ for every $k\in\N$.
For $k\in\N_{\infty}:=\N\cup\{\infty\}$ the functions $f\colon\Omega\to E$ which are $\mathcal{C}^{k}$ form a linear space which
is denoted by $\mathcal{C}^{k}(\Omega,E)$. 
For $\beta\in\N_{0}^{d}$ with $|\beta|:=\sum_{n=1}^{d}\beta_{n}\leq k$ and a function $f\colon\Omega\to E$ 
on an open set $\Omega\subset\R^{d}$ to an lcHs $E$ we set $(\partial^{\beta_{n}})^{E}f:=f$ if $\beta_{n}=0$, and
\[
(\partial^{\beta_{n}})^{E}f(x)
:=\underbrace{(\partial^{e_{n}})^{E}\cdots(\partial^{e_{n}})^{E}}_{\beta_{n}\text{-times}}f(x)
\]
if $\beta_{n}\neq 0$ and the right-hand side exists in $E$ for every $x\in\Omega$.
Further, we define 
\[
(\partial^{\beta})^{E}f(x)
=:\bigl((\partial^{\beta_{1}})^{E}\cdots(\partial^{\beta_{d}})^{E}\bigr)f(x)
\]
if the right-hand side exists in $E$ for every $x\in\Omega$ and set $f^{(\beta)}:=(\partial^{\beta})^{E}f$ if $d=1$.
\section{Schauder decomposition}
Let us start with our main theorem on Schauder decompositions of $\varepsilon$-products.

\begin{thm}\label{thm:schauder_decomp}
Let $F$ and $E$ be lcHs, $(f_{n})_{n\in\N}$ an equicontinuous Schauder basis of $F$ 
with associated coefficient functionals $(\lambda_{n})_{n\in\N}$ and set $Q_{n}\colon F\to F$, 
$Q_{n}(f):=\lambda_{n}(f)f_{n}$ for every $n\in\N$. Then the following holds.
\begin{enumerate}
\item [a)] The sequence $(P_{k})_{k\in\N}$ given by $P_{k}:=\bigl(\sum_{n=1}^{k}Q_{n}\bigr)\varepsilon\id_{E}$ is 
a Schauder decomposition of $F\varepsilon E$.
\item [b)] Each $u\in F\varepsilon E$ has the series representation
\[
u(f')=\sum_{n=1}^{\infty}u(\lambda_{n})f'(f_{n}),\quad f'\in F'.
\]
\item [c)] $F\otimes E$ is sequentially dense in $F\varepsilon E$. 
\end{enumerate}
\end{thm}
\begin{proof}
Since $(f_{n})$ is a Schauder basis of $F$, the sequence $(\sum_{n=1}^{k}Q_{n})$ converges weakly to $\id_{F}$. 
Thus we deduce from the equicontinuity of $(f_{n})$ that $(\sum_{n=1}^{k}Q_{n})$ converges to $\id_{F}$ in $L_{\kappa}(F,F)$ 
by \cite[Theorem 8.5.1 (b), p.\ 156]{Jarchow}. For $f'\in F'$ and $f\in F$ holds
\begin{align*}
  (Q_{n}^{t}\circ Q_{m}^{t})(f')(f)
&=Q_{m}^{t}(f')(Q_{n}(f))=Q_{m}^{t}(f')(\lambda_{n}(f)f_{n})=f'(\lambda_{m}(\lambda_{n}(f)f_{n})f_{m})\\
&=\lambda_{m}(f_{n})\lambda_{n}(f)f'(f_{m})
 =\begin{cases}
   \lambda_{n}(f)f'(f_{n})&,\; m=n,\\
   0&,\; m\neq n,
  \end{cases}
\end{align*}
due to the uniqueness of the coefficient functionals $(\lambda_{n})$ (see \cite[14.2.1 Proposition, p.\ 292]{Jarchow}) and
it follows for $k,j\in\N$ that 
\[
(\sum_{n=1}^{j}Q_{n}^{t}\circ\sum_{m=1}^{k}Q_{n}^{t})(f')(f)=\sum_{n=1}^{\min(j,k)}\lambda_{n}(f)f'(f_{n})
=\sum_{n=1}^{\min(j,k)}Q_{n}^{t}(f')(f).
\]
This implies that
\[
(P_{k}P_{j})(u)=u\circ\sum_{n=1}^{j}Q_{n}^{t}\circ\sum_{m=1}^{k}Q_{n}^{t}
=u\circ\sum_{n=1}^{\min(j,k)}Q_{n}^{t}=P_{\min(j,k)}(u)
\]
for all $u\in F\varepsilon E$. If $k\neq j$, w.l.o.g.\ $k>j$, we choose $x\in E$, $x\neq 0$, and 
define $u_{k,x}\colon F'\to E$ by $u_{k,x}(f'):=f'(f_{k})x$. Then $u_{k,x}\in F\varepsilon E$ and 
\[
(P_{k}-P_{j})(u_{k,x})=\sum_{n=j+1}^{k}u_{k,x}\circ Q_{n}^{t}=u_{k,x}\neq 0
\]
since 
\[
(u_{k,x}\circ Q_{n}^{t})(f')=u_{k,x}(\lambda_{n}(\cdot)f'(f_{n}))=\lambda_{n}(f_{k})f'(f_{n})x
 =\begin{cases}
   f'(f_{k})x&,\; n=k,\\
   0&,\; n\neq k.
  \end{cases}
\]
It remains to prove that for each $u\in F\varepsilon E$ 
\[
\lim_{k\to\infty}P_{k}(u)=u
\]
in $F\varepsilon E$. Let $(q_{\beta})_{\beta\in\mathfrak{B}}$ denote the system of seminorms inducing 
the locally convex topology of $F$. Let $u\in F\varepsilon E$ and $\alpha\in\mathfrak{A}$. 
Due to the continuity of $u$ there are an absolutely convex compact set $K=K(u,\alpha)\subset F$ 
and $C_{0}=C_{0}(u,\alpha)>0$ such that for each $f'\in F'$ we have 
\begin{align*}
 p_{\alpha}\bigl((P_{k}(u)-u)(f')\bigr)
&=p_{\alpha}\bigl(u\bigl((\sum_{n=1}^{k}Q_{n}^{t}-\id_{F'})(f')\bigr)\bigr)
\leq C_{0}\sup_{f\in K}\bigl|(\sum_{n=1}^{k}Q_{n}^{t}-\id_{F'})(f')(f)\bigr|\\
&=C_{0}\sup_{f\in K}\bigl|f'(\sum_{n=1}^{k}Q_{n}f-f)\bigr|.
\end{align*}
Let $V$ be an absolutely convex $0$-neighbourhood in $F$. As a consequence of the equicontinuity of 
the polar $V^{\circ}$ there are $C_{1}>0$ and $\beta\in\mathfrak{B}$ such that 
\[ 
\sup_{f'\in V^{\circ}}p_{\alpha}\bigl((P_{k}(u)-u)(f')\bigr)
\leq C_{0}C_{1}\sup_{f\in K}q_{\beta}(\sum_{n=1}^{k}Q_{n}f-f).
\]
In combination with the convergence of $(\sum_{n=1}^{k}Q_{n})$ to $\id_{F}$ in $L_{\kappa}(F,F)$ 
this yields the convergence of $(P_{k}(u))$ to $u$ in $F\varepsilon E$ and settles part a). 

Let us turn to b) and c). Since 
\[
P_{k}(u)(f')=u\bigl(\sum_{n=1}^{k}Q_{n}^{t}(f')\bigr)=\sum_{n=1}^{k}u(\lambda_{n})f'(f_{n})
\]
for every $f'\in F'$, we note that the range of $P_{k}(u)$ is contained in $\Span(u(\lambda_{n})\;|\;1\leq n\leq k)$ 
for each $u\in F\varepsilon E$ and $k\in\N$. Hence $P_{k}(u)$ has finite rank and thus belongs to $F\otimes E$ 
implying the sequential density of $F\otimes E$ in $F\varepsilon E$ and the desired series representation by part a).
\end{proof}

\begin{rem}\label{rem:schauder_decomp}
If $F$ and $E$ are complete, we have under the assumption of \prettyref{thm:schauder_decomp} that 
$F\,\widehat{\otimes}_{\varepsilon}E\cong F\varepsilon E$ by c) since $F\varepsilon E$ is complete 
by \cite[Satz 10.3, p.\ 234]{Kaballo} and $F\,\widehat{\otimes}_{\varepsilon}E$ is the closure 
of $F\otimes E$ in $F\varepsilon E$. 
Thus each element of $F\,\widehat{\otimes}_{\varepsilon}E$ has a series representation. 
\end{rem}

Let us apply the preceding theorem to spaces of Lebesgue integrable functions. 
We consider the measure space $([0,1], \mathscr{L}([0,1]), \lambda)$ of Lebesgue measurable sets 
and use the notation $\mathcal{L}^{p}[0,1]$ for 
the space of (equivalence classes) of Lebesgue $p$-integrable functions on $[0,1]$.
The Haar system $h_{n}\colon [0,1]\to\R$, $n\in\N$, given by 
$h_{1}(x):=1$ for all $x\in [0,1]$ and 
\[
h_{2^{k}+j}(x):=
\begin{cases}
\phantom{-}1 &, (2j-2)/2^{k+1}\leq x<(2j-1)/2^{k+1},\\
-1 &, (2j-1)/2^{k+1}\leq x<2j/2^{k+1},\\
\phantom{-}0 &, \text{else},
\end{cases}
\]
for $k\in\N_{0}$ and $1\leq j\leq 2^{k}$ forms a Schauder basis 
of $\mathcal{L}^{p}[0,1]$ for every $1\leq p<\infty$ and the associated coefficient 
functionals are 
\[
\lambda_{n}(f):=\int_{[0,1]}f(x)h_{n}(x)\d\lambda(x), \quad f\in \mathcal{L}^{p}[0,1],\;n\in\N,
\]
(see \cite[Satz I, p.\ 317]{Schauder1928}). Because $\mathcal{L}^{p}[0,1]$ is Banach space and thus barrelled, 
its Schauder basis $(h_{n})$ is equicontinuous and 
we directly obtain from \prettyref{thm:schauder_decomp} the following corollary. 

\begin{cor}\label{cor:L_p_>1}
Let $E$ be an lcHs and $1\leq p<\infty$. Then $(\sum_{n=1}^{k}\lambda_{n}(\cdot)h_{n}\varepsilon\id_{E})_{k\in\N}$ is 
a Schauder decomposition of $\mathcal{L}^{p}[0,1]\varepsilon E$ and for
each $u\in\mathcal{L}^{p}[0,1]\varepsilon E$ holds
\[
u(f')=\sum_{n=1}^{\infty}u(\lambda_{n})f'(f_{n}),\quad f'\in \mathcal{L}^{p}[0,1]'.
\]
\end{cor}

Defining $\mathcal{L}^{p}([0,1],E):=\mathcal{L}^{p}[0,1]\varepsilon E$, we can read the corollary above as a statement 
on series representations in the vector-valued version of $\mathcal{L}^{p}[0,1]$. 
However, in many cases of spaces $\mathcal{F}(\Omega)$ of scalar-valued functions there 
is a more natural way to define the vector-valued version $\mathcal{F}(\Omega,E)$ of $\mathcal{F}(\Omega)$, 
see for example the space of holomorphic functions from the introduction.
If $\mathcal{F}(\Omega)\varepsilon E$ and $\mathcal{F}(\Omega,E)$ are isomorphic 
via the map $S$ from the introduction, we can translate \prettyref{thm:schauder_decomp} 
to the more natural setting $\mathcal{F}(\Omega,E)$ 
which motivates the following definition.

\begin{defn}[$\varepsilon$-compatible]\label{def:compatible}
Let $\Omega$ be a set and $E$ an lcHs. If $\mathcal{F}(\Omega)\subset \K^{\Omega}$ and 
$\mathcal{F}(\Omega,E)\subset E^{\Omega}$ are lcHs such 
that $\delta_{x}\in\mathcal{F}(\Omega)'$ for all $x\in\Omega$ and
\[
S\colon \mathcal{F}(\Omega)\varepsilon E\to \mathcal{F}(\Omega,E),\; u\longmapsto [x\mapsto u(\delta_{x})],
\]  
is an isomorphism, then we say that $\mathcal{F}(\Omega)$ and $\mathcal{F}(\Omega,E)$ are $\varepsilon$-compatible.
\end{defn}

If we want to emphasise the dependence of $S$ on $\mathcal{F}(\Omega)$, we write $S_{\mathcal{F}(\Omega)}$. 
Several sufficient conditions for $S$ being an isomorphism are given in \cite[3.17 Theorem, p.\ 12]{kruse2017}. 

\begin{rem}\label{rem:identification_subspaces}
Let $\Omega$ be a set and $E$ an lcHs. If $\mathcal{F}(\Omega)\subset \K^{\Omega}$ and 
$\mathcal{F}(\Omega,E)\subset E^{\Omega}$ are lcHs such 
that $\delta_{x}\in\mathcal{F}(\Omega)'$ for all $x\in\Omega$ and
\[
S\colon \mathcal{F}(\Omega)\varepsilon E\to \mathcal{F}(\Omega,E),\; u\longmapsto [x\mapsto u(\delta_{x})],
\]  
is an isomorphism into, then we get by identification of isomorphic subspaces
 \[
  \mathcal{F}(\Omega)\otimes_{\varepsilon} E\subset \mathcal{F}(\Omega)\varepsilon E \subset \mathcal{F}(\Omega,E)
 \]
 and the embedding $ \mathcal{F}(\Omega)\otimes E\hookrightarrow \mathcal{F}(\Omega,E)$ is given by 
 $f\otimes e\mapsto [x\mapsto f(x)e]$.
\end{rem}
\begin{proof}
The inclusions obviously hold and $\mathcal{F}(\Omega)\varepsilon E$ and $\mathcal{F}(\Omega,E)$ 
induce the same topology on $\mathcal{F}(\Omega)\otimes E$. 
Further, we have
 \[
 f\otimes e \overset{\Theta}{\longmapsto}[y\mapsto y(f)e]
 \overset{S}{\longmapsto}[x\longmapsto [y\mapsto y(f) e](\delta_{x})]
 =[x\mapsto f(x)e].
 \]
\end{proof}

\begin{cor}\label{cor:schauder_decomp}
Let $\mathcal{F}(\Omega)$ and $\mathcal{F}(\Omega,E)$ be $\varepsilon$-compatible spaces, 
$(f_{n})_{n\in\N}$ an equi\-continuous Schauder basis of $\mathcal{F}(\Omega)$ 
with associated coefficient functionals $(\lambda_{n}^{\K})_{n\in\N}$ and let 
$\lambda_{n}^{E}\colon \mathcal{F}(\Omega,E)\to E$ such that 
\begin{equation}\label{eq:schauder_decomp}
\lambda_{n}^{E}(S(u))=u(\lambda_{n}^{\K}),\quad u\in\mathcal{F}(\Omega)\varepsilon E,
\end{equation}
for all $n\in\N$. Set $Q_{n}^{E}\colon \mathcal{F}(\Omega,E)\to \mathcal{F}(\Omega,E)$, 
$Q_{n}^{E}(f):=\lambda_{n}^{E}(f)f_{n}$ for every $n\in\N$. Then the following holds.
\begin{enumerate}
\item [a)] The sequence $(P_{k}^{E})_{k\in\N}$ given by $P_{k}^{E}:=\sum_{n=1}^{k}Q_{n}^{E}$ is 
a Schauder decomposition of $\mathcal{F}(\Omega,E)$.
\item [b)] Each $f\in \mathcal{F}(\Omega,E)$ has the series representation
\[
f=\sum_{n=1}^{\infty}\lambda_{n}^{E}(f)f_{n}.
\]
\item [c)] $\mathcal{F}(\Omega)\otimes E$ is sequentially dense in $\mathcal{F}(\Omega,E)$. 
\end{enumerate}
\end{cor}
\begin{proof}
For each $u\in\mathcal{F}(\Omega)\varepsilon E$ and $x\in\Omega$ we note that
 \begin{align*}
    (S\circ P_{k})(u)(x)
  &=S\bigl(\sum_{n=1}^{k}Q_{n}^{t}(\delta_{x})\bigr) 
   =u\bigl(\sum_{n=1}^{k}\lambda_{n}^{\K}(\cdot)f_{n}(x)\bigr) 
   =\sum_{n=1}^{k}u(\lambda_{n}^{\K})f_{n}(x)\\
  &=\sum_{n=1}^{k}\lambda_{n}^{E}(S(u))f_{n}(x)
   =(P_{k}^{E}\circ S)(u)(x)
 \end{align*}
which means that $S\circ P_{k}=P_{k}^{E}\circ S$. This implies part a) and b) 
by \prettyref{thm:schauder_decomp} a) since $S$ is an isomorphism. 
Part c) is a direct consequence of \prettyref{thm:schauder_decomp} c) and 
the isomorphy $\mathcal{F}(\Omega)\varepsilon E\cong\mathcal{F}(\Omega,E)$.
\end{proof}

In the preceding corollary we used the isomorphism $S$ to obtain a Schauder decomposition. 
On the other hand, if $S$ is an isomorphism into which is often the case (see \cite[3.9 Theorem, p.\ 9]{kruse2017}), 
we can use a Schauder decomposition of $\mathcal{F}(\Omega,E)$ to 
prove the surjectivity of $S$.

\begin{prop}\label{prop:reverse_Schauder}
Let $\Omega$ be a set and $E$ an lcHs. Let $\mathcal{F}(\Omega)\subset \K^{\Omega}$ and 
$\mathcal{F}(\Omega,E)\subset E^{\Omega}$ be lcHs such 
that $\delta_{x}\in\mathcal{F}(\Omega)'$ for all $x\in\Omega$ and
\[
S\colon \mathcal{F}(\Omega)\varepsilon E\to \mathcal{F}(\Omega,E),\; u\longmapsto [x\mapsto u(\delta_{x})],
\]  
is an isomorphism into.
Let there be $(f_{n})_{n\in\N}$ in $\mathcal{F}(\Omega)$ and 
for every $f\in \mathcal{F}(\Omega,E)$ a sequence $(\lambda_{n}^{E}(f))_{n\in\N}$ in $E$ such that
 \[
  f=\sum_{n=1}^{\infty}\lambda_{n}^{E}(f)f_{n}, \quad f\in \mathcal{F}(\Omega,E).
 \]
Then the following holds.
\begin{enumerate}
\item [a)] $\mathcal{F}(\Omega)\otimes E$ is sequentially dense in $\mathcal{F}(\Omega,E)$.
\item [b)] If $\mathcal{F}(\Omega)$ and $E$ are sequentially complete, then
 \[
  \mathcal{F}(\Omega,E)\cong\mathcal{F}(\Omega)\varepsilon E.
 \] 
\item [c)] If $\mathcal{F}(\Omega)$ and $E$ are complete, then
 \[
  \mathcal{F}(\Omega,E)\cong\mathcal{F}(\Omega)\varepsilon E\cong\mathcal{F}(\Omega)\widehat{\otimes}_{\varepsilon}E.
 \]
\end{enumerate}
\end{prop}
\begin{proof}
Let $f\in \mathcal{F}(\Omega,E)$ and observe that
\[
P^{E}_{k}(f):=\sum_{n=1}^{k}\lambda_{n}^{E}(f)f_{n}=\sum_{n=1}^{k}f_{n}\otimes \lambda_{n}^{E}(f)\in \mathcal{F}(\Omega)\otimes E
\]
for every $k\in\N$ by \prettyref{rem:identification_subspaces}.
Due to our assumption we have the convergence $P^{E}_{k}(f)\to f$ in $\mathcal{F}(\Omega,E)$. 
Thus $\mathcal{F}(\Omega)\otimes E$ is sequentially dense in 
$\mathcal{F}(\Omega,E)$. 

Let us turn to part b). If $\mathcal{F}(\Omega)$ and $E$ are sequentially complete, 
then $\mathcal{F}(\Omega)\varepsilon E$ is sequentially complete 
by \cite[Satz 10.3, p.\ 234]{Kaballo}. Since $S$ is an isomorphism into and
\[
S(\Theta(\sum_{n=q}^{k}f_{n}\otimes \lambda_{n}^{E}(f)))=\sum_{n=q}^{k}\lambda_{n}^{E}(f)f_{n}
\]
for all $k,q\in\N$, $k>q$, we get that $(\Theta(\sum_{n=1}^{k}f_{n}\otimes \lambda_{n}^{E}(f))$ is a Cauchy sequence 
in $\mathcal{F}(\Omega)\varepsilon E$ and thus convergent. Hence we deduce that
\[
S(\lim_{k\to\infty}\Theta(\sum_{n=1}^{k}f_{n}\otimes \lambda_{n}^{E}(f)))
=\lim_{k\to\infty}\sum_{n=1}^{k}(S\circ\Theta)(f_{n}\otimes \lambda_{n}^{E}(f))
=\sum_{n=1}^{\infty}\lambda_{n}^{E}(f)f_{n}=f
\]
which proves the surjectivity of $S$. 

If $\mathcal{F}(\Omega)$ and $E$ are complete, then $\mathcal{F}(\Omega)\widehat{\otimes}_{\varepsilon}E$ 
is the closure of $\mathcal{F}(\Omega)\otimes_{\varepsilon}E$ in the complete space 
$\mathcal{F}(\Omega)\varepsilon E$ by \cite[Satz 10.3, p.\ 234]{Kaballo}.
As $\lim_{k\to\infty}\Theta(\sum_{n=1}^{k}f_{n}\otimes \lambda_{n}^{E}(f))$ is an element of the closure, 
we obtain part c). 
\end{proof}
\section{Applications}
\section*{Sequence spaces}
For our first application we recall the definition of some sequence spaces.
A matrix $A:=\left(a_{k,j}\right)_{k,j\in\N}$ of non-negative numbers is called K\"othe matrix if 
it fulfils:
\begin{enumerate}
	\item [(1)] $\forall\;k\in\N\;\exists\;j\in\N:\;a_{k,j}>0$,
	\item [(2)] $\forall\;k,j\in\N:\;a_{k,j}\leq a_{k,j+1}$.
\end{enumerate}
For an lcHs $E$ we define the K\"othe space
\[
\lambda^{\infty}(A,E):=\{x=(x_{k})\in E^{\N}\;|\;\forall\;j\in\N,\,\alpha\in \mathfrak{A}:\; 
|x|_{j,\alpha}:=\sup_{k\in\N}p_{\alpha}(x_{k})a_{k,j}<\infty\}
\]
and the topological subspace
\[
 c_{0}(A,E):=\{x=(x_{k})\in E^{\N}\;|\;\forall\;j\in\N:\;\lim_{k\to\infty}x_{k}a_{k,j}=0\}.
\]
In particular, the space $c_{0}(\N,E)$ of null-sequences in $E$ is obtained as 
$c_{0}(\N,E)=c_{0}(A,E)$ with $a_{k,j}:=1$ for all $k,j\in\N$. 
The space of convergent sequences in $E$ is defined by
\[
c(\N,E):=\{x\in E^{\N}\;|\; x=(x_{k})\;\text{converges in}\;E\}
\]
and equipped with the system of seminorms
\[
|x|_{\alpha}:=\sup_{k\in\N}p_{\alpha}(x_{k}),\quad x\in c(\N,E),
\]
for $\alpha\in\mathfrak{A}$. We define the spaces of $E$-valued rapidely decreasing sequences 
which we need for our subsection on Fourier expansion by 
\[
s(\Omega,E):=\{x=(x_{k})\in E^{\Omega}\;|\;\forall\;j\in\N,\,\alpha\in\mathfrak{A}:\; 
|x|_{j,\alpha}:=\sup_{k\in\Omega}p_{\alpha}(x_{k})(1+|k|^{2})^{j/2}<\infty\}
\]
with $\Omega=\N^{d}$, $\N_{0}^{d}$, $\Z^{d}$. 
Furthermore, we equip the space $E^{\N}$ with the system of seminorms given by 
\[
 \|x\|_{l,\alpha}:=\sup_{k\in\N}p_{\alpha}(x_{k})\chi_{\{1,\ldots,l\}}(k),\quad x=(x_{k})\in E^{\N},
\]
for $l\in\N$ and $\alpha\in\mathfrak{A}$. 
For a non-empty set $\Omega$ we define for $n\in\Omega$ the $n$-th unit function by 
\[
 \varphi_{n,\Omega}\colon \Omega\to \K,\;
 \varphi_{n,\Omega}(k):=\begin{cases} 1 &,\;k=n, \\ 0 &,\;\text{else},\end{cases}
\]
and we simply write $\varphi_{n}$ instead of $\varphi_{n,\Omega}$ if no confusion seems to be likely. Further, we set
$
 \varphi_{\infty}\colon \N\to \K,\;
 \varphi_{\infty}(k):=1,
$
and $x_{\infty}:=\delta_{\infty}(x):=\lim_{k\to\infty}x_{k}$ for $x\in c(\N,E)$.  
For series representations of the elements in these sequence spaces we do not need \prettyref{cor:schauder_decomp} 
due to the subsequent proposition but we can use the representation to obtain the surjectivity of $S$ for sequentially complete $E$.

\begin{prop}\label{prop:vector_valued_seq_spaces}
 Let $E$ be an lcHs and $\ell\mathcal{V}(\Omega,E)$ one of the spaces $c_{0}(A,E)$, $E^{\N}$, 
 $s(\N^{d},E)$, $s(\N_{0}^{d},E)$ or $s(\Z^{d},E)$. 
 \begin{enumerate}
 \item [a)] Then $(\sum_{n\in\Omega,|n|\leq k}\delta_{n} \varphi_{n})_{k\in\N}$ is a Schauder decomposition 
 of $\ell\mathcal{V}(\Omega,E)$ and 
  \[
  x=\sum_{n\in\Omega} x_{n}\varphi_{n},\quad x\in \ell\mathcal{V}(\Omega,E).
  \]
 \item [b)] Then $(\delta_{\infty}\varphi_{\infty}+\sum_{n=1}^{k}(\delta_{n}-\delta_{\infty}) \varphi_{n})_{k\in\N}$ 
 is a Schauder decomposition of $c(\N,E)$ and 
 \[
  x=x_{\infty}\varphi_{\infty} +\sum_{n=1}^{\infty}(x_{n}-x_{\infty})\varphi_{n},\quad x\in c(\N,E).
 \]
 \end{enumerate}
\end{prop}
\begin{proof}
Let us begin with a). For $x=(x_{n})\in \ell\mathcal{V}(\Omega,E)$ let $(P_{k}^{E})$ be the sequence 
in $\ell\mathcal{V}(\Omega,E)$ given by $P_{k}^{E}(x):=\sum_{|n|\leq k} x_{n}\varphi_{n}$. 
It is easy to see that $P_{k}^{E}$ is a continuous projection on $\ell\mathcal{V}(\Omega,E)$, 
$P_{k}^{E}P_{j}^{E}=P_{\min(k,j)}^{E}$ for all $k,j\in\N$ and $P_{k}^{E}\neq P_{j}^{E}$ for $k\neq j$.
Let $\varepsilon>0$, $\alpha\in\mathfrak{A}$ and $j\in\N$. 
For $x\in c_{0}(A,E)$ there is $N_{0}\in\N$ such that 
$p_{\alpha}(x_{n}a_{n,j})<\varepsilon$ for all $n\geq N_{0}$. 
Hence we have for $x\in c_{0}(A,E)$
\[
|x-P_{k}^{E}(x)|_{j,\alpha}=\sup_{n>k}p_{\alpha}(x_{n})a_{n,j}\leq\sup_{n\geq N_{0}}p_{\alpha}(x_{n})a_{n,j}\leq\varepsilon
\]
for all $k\geq N_{0}$. For $x\in E^{\N}$ and $l\in\N$ we have 
\[
\left\|x-P_{k}^{E}(x)\right\|_{l,\alpha}=0<\varepsilon
\]
for all $k\geq l$. For $x\in s(\Omega,E)$, $\Omega=\N^{d}$, $\N_{0}^{d}$, $\Z^{d}$, 
we notice that there is $N_{1}\in\N$ such that for all $n\in\Omega$ with $|n|\geq N_{1}$ we have
\[
\frac{(1+|n|^{2})^{j/2}}{(1+|n|^{2})^{j}}=(1+|n|^{2})^{-j/2}<\varepsilon.
\]
Thus we deduce for $|n|\geq N_{1}$
\[
p_{\alpha}(x_{n})(1+|n|^{2})^{j/2}< \varepsilon p_{\alpha}(x_{n})(1+|n|^{2})^{j}
\leq \varepsilon |x|_{2j,\alpha}
\]
and hence
\[
|x-P_{k}^{E}(x)|_{j,\alpha}=\sup_{|n|>k}p_{\alpha}(x_{n})a_{n,j}\leq\sup_{|n|\geq N_{1}}p_{\alpha}(x_{n})(1+|n|^{2})^{j/2}
\leq\varepsilon |x|_{2j,\alpha}
\]
for all $k\geq N_{1}$. Therefore $(P_{k}^{E}(x))$ converges to $x$ in $\ell\mathcal{V}(\Omega,E)$ and 
\[
x=\lim_{k\to\infty}P_{k}^{E}(x)=\sum_{n\in\Omega}x_{n}\varphi_{n}.
\]

Now, we turn to b). For $x=(x_{n})\in c(\N,E)$ let $(\widetilde{P}_{k}^{E}(x))$ be the sequence in $c(\N,E)$ given by 
$\widetilde{P}_{k}^{E}(x):=x_{\infty}\varphi_{\infty}+\sum_{n=1}^{k} (x_{n}-x_{\infty}) \varphi_{n}$. 
Again, it is easy to see that $\widetilde{P}_{k}^{E}$ is a continuous projection on $c(\N,E)$, 
$\widetilde{P}_{k}^{E}\widetilde{P}_{j}^{E}=\widetilde{P}_{\min(k,j)}^{E}$ for all $k,j\in\N$ 
and $\widetilde{P}_{k}^{E}\neq\widetilde{P}_{j}^{E}$ for $k\neq j$.
Let $\varepsilon>0$ and $\alpha\in\mathfrak{A}$. Then there is $N_{2}\in\N$ such that 
$p_{\alpha}(x_{n}-x_{\infty})<\varepsilon$ for all $n\geq N_{2}$. Thus we obtain 
\[
|x-\widetilde{P}_{k}^{E}(x)|_{\alpha}=\sup_{n>k}p_{\alpha}(x_{n}-x_{\infty})\leq \sup_{n\geq N_{2}}p_{\alpha}(x_{n}-x_{\infty})
\leq\varepsilon
\]
for all $k\geq N_{2}$ implying that $(\widetilde{P}_{k}^{E}(x))$ converges to $x$ in $c(\N,E)$ and 
\[
x=\lim_{k\to\infty}\widetilde{P}_{k}^{E}(x)=x_{\infty}\varphi_{\infty}+\sum_{n=1}^{\infty}(x_{n}-x_{\infty})\varphi_{n}.
\]
\end{proof}

\begin{thm}\label{thm:sequence.spaces}
 Let $E$ be a sequentially complete lcHs and $\ell\mathcal{V}(\Omega,E)$ one of the spaces $c_{0}(A,E)$, $E^{\N}$, 
 $s(\N^{d},E)$, $s(\N_{0}^{d},E)$ or $s(\Z^{d},E)$. Then
 \[
  (i)\;\ell\mathcal{V}(\Omega,E)\cong \ell\mathcal{V}(\Omega)\varepsilon E,\quad (ii)\;c(\N,E)\cong c(\N)\varepsilon E.
 \]
\end{thm}
\begin{proof} 
The map $S_{\ell\mathcal{V}(\Omega)}$ is an isomorphism into by \cite[3.9 Theorem, p.\ 9]{kruse2017} and 
\cite[4.13 Proposition, p.\ 23]{kruse2017}.
Considering $c(\N,E)$, we observe that for $x\in c(\N)$ 
\[
\delta_{n}(x)=x_{n}\to x_{\infty}=\delta_{\infty}(x)
\]
which implies the concergence $\delta_{n}\to \delta_{\infty}$ 
in $c(\N)_{\gamma}'$ by the Banach-Steinhaus theorem since $c(\N)$ is a Banach space. 
Hence we get 
\[
u(\delta_{\infty})=\lim_{n\to\infty}u(\delta_{n})
=\lim_{n\to\infty}S(u)(n)=\delta_{\infty}(S(u))
\]
for every $u\in c(\N)\varepsilon E$ which implies that $S_{c(\N)}$ is an isomorphism into 
by \cite[3.9 Theorem, p.\ 9]{kruse2017}.
From \prettyref{prop:vector_valued_seq_spaces} and \prettyref{prop:reverse_Schauder} we deduce our statement.
\end{proof}

\section*{Continuous and differentiable functions on a closed interval}

We start with continuous functions on compact sets. 
We recall the following definition from \cite[p.\ 259]{J.Voigt}. A locally convex Hausdorff space is said 
to have the metric convex compactness property (metric ccp) if 
the closure of the absolutely convex hull of every metrisable compact set is compact. 
In particular, every sequentially complete space has metric ccp and this implication is strict 
(see \cite[p.\ 3-4]{kruse2017} and the references therein).
Let $E$ be an lcHs, $\Omega\subset\R^{d}$ compact and denote by $\mathcal{C}(\Omega,E):=\mathcal{C}^{0}(\Omega,E)$ 
the space of continuous functions from $\Omega$ to $E$. We equip $\mathcal{C}(\Omega,E)$ with the system of seminorms given by
\[
|f|_{0,\alpha}:=\sup_{x\in\Omega}p_{\alpha}(f(x)),\quad f\in\mathcal{C}(\Omega,E),
\]
for $\alpha\in\mathfrak{A}$.
We want to apply our preceding results to intervals.
Let $-\infty<a<b<\infty$ and $T:=(t_{j})_{0\leq j\leq n}$ be a partition of 
the interval $[a,b]$, i.e.\ $a=t_{0}<t_{1}<\ldots<t_{n}=b$. 
The hat functions $h_{t_{j}}^{T}\colon [a,b]\to\R$ for the partition $T$ are given by 
\[
h_{t_{j}}^{T}(x):=
\begin{cases}
\frac{x-t_{j}}{t_{j}-t_{j-1}} &, t_{j-1}\leq x\leq t_{j},\\
\frac{t_{j+1}-x}{t_{j+1}-t_{j}} &,t_{j}< x\leq t_{j+1},\\
0 &, \text{else},
\end{cases}
\]
for $2\leq j\leq n-1$ and 
\[
h_{a}^{T}(x):=
\begin{cases}
\frac{t_{1}-x}{t_{1}-a} &, a\leq x\leq t_{1},\\
0 &, \text{else},
\end{cases}
\quad 
h_{b}^{T}(x):=
\begin{cases}
\frac{x-t_{n-1}}{b-t_{n-1}} &, t_{n-1}\leq x\leq b,\\
0 &, \text{else}.
\end{cases}
\]
Let $\mathcal{T}:=(t_{n})_{n\in\N_{0}}$ be a dense sequence in $[a,b]$ with $t_{0}=a$, $t_{1}=b$ 
and $t_{n}\neq t_{m}$ for $n\neq m$. For $T^{n}:=\{t_{0},\ldots,t_{n}\}$ there is a (unique) enumeration 
$\sigma\colon \{0,\ldots,n\}\to T^{n}$ such that $T_{n}:=(t_{\sigma(j)})_{0\leq j\leq n}$ 
is a partition of $[a,b]$. 
The functions $\varphi^{\mathcal{T}}_{0}:=h^{T_{1}}_{t_{0}}$, 
$\varphi^{\mathcal{T}}_{1}:=h^{T_{1}}_{t_{1}}$ and $\varphi^{\mathcal{T}}_{n}:=h^{T_{n}}_{t_{n}}$ 
for $n\geq 2$ are called Schauder hat functions for the sequence $\mathcal{T}$ and form a Schauder basis 
of $\mathcal{C}([a,b])$ with associated coefficient functionals given by 
$\lambda^{\K}_{0}(f):=f(t_{0})$, $\lambda^{\K}_{1}(f):=f(t_{1})$ and 
\[
\lambda^{\K}_{n+1}(f):=f(t_{n+1})
-\sum_{k=0}^{n}\lambda^{\K}_{k}(f)\varphi^{\mathcal{T}}_{k}(t_{n+1}),\quad f\in\mathcal{C}([a,b]),\; n\geq 1,
\]
by \cite[2.3.5 Proposition, p.\ 29]{Semadeni1982}. Looking at the coefficient functionals, we 
see that the right-hand sides even make sense if $f\in\mathcal{C}([a,b],E)$ and thus we define 
$\lambda^{E}_{n}$ on $\mathcal{C}([a,b],E)$ for $n\in\N_{0}$ accordingly. 

\begin{thm}\label{thm:cont.func.interval}
 Let $E$ be an lcHs with metric ccp and $\mathcal{T}:=(t_{n})_{n\in\N_{0}}$ a dense sequence in $[a,b]$ 
 with $t_{0}=a$, $t_{1}=b$ and $t_{n}\neq t_{m}$ for $n\neq m$. 
 Then $(\sum_{n=0}^{k}\lambda_{k}^{E}\varphi^{\mathcal{T}}_{n})_{k\in\N_{0}}$ is a Schauder decomposition of 
 $\mathcal{C}([a,b],E)$ and 
 \[
 f=\sum_{n=0}^{\infty}\lambda^{E}_{n}(f)\varphi^{\mathcal{T}}_{n},\quad f\in \mathcal{C}([a,b],E).
 \]
\end{thm}
\begin{proof}
The spaces $\mathcal{C}([a,b])$ and $\mathcal{C}([a,b],E)$ are $\varepsilon$-compatible by 
\cite[5.4 Example, p.\ 25]{kruse2017} if $E$ has metric ccp. 
$\mathcal{C}([a,b])$ is a Banach space and thus barrelled implying 
that its Schauder basis $(\varphi^{\mathcal{T}}_{n})$ is equicontinuous. 
We note that for all $u\in\mathcal{C}([a,b])\varepsilon E$ and $x\in [a,b]$
\[
\lambda_{n}^{E}(S(u))(x)=u(\delta_{t_{n}})=u(\lambda^{\K}_{n}),\quad n\in\{0,1\},
\] 
and by induction 
\begin{align*}
\lambda_{n+1}^{E}(S(u))(x)&=u(\delta_{t_{n+1}})-\sum_{k=0}^{n}\lambda^{E}_{k}(S(u))\varphi^{\mathcal{T}}_{k}(t_{n+1})
=u(\delta_{t_{n+1}})-\sum_{k=0}^{n}u(\lambda^{\K}_{k})\varphi^{\mathcal{T}}_{k}(t_{n+1})\\
&=u(\lambda_{n+1}^{\K}),\quad n\geq 1.
\end{align*}
Thus \eqref{eq:schauder_decomp} is fulfilled proving our claim by \prettyref{cor:schauder_decomp}.
\end{proof}

If $a=0$, $b=1$ and $\mathcal{T}$ is the sequence of dyadic numbers given in 
\cite[2.1.1 Definitions, p.\ 21]{Semadeni1982}, then $(\varphi^{\mathcal{T}}_{n})$ is the 
so-called Faber-Schauder system. Using the Schauder basis and coefficient functionals of 
the space $\mathcal{C}_{0}(\R)$ of continuous functions vanishing at infinity given 
in \cite[2.7.1, p.\ 41-42]{Semadeni1982} and \cite[2.7.4 Corollary, p.\ 43]{Semadeni1982} 
a corresponding, weaker result for the $E$-valued 
counterpart $\mathcal{C}_{0}(\R,E)$ holds as well by a similar argumentation. 
Another corresponding result holds for the space $C^{[\gamma]}_{0,0}([0,1],E)$, $0<\gamma<1$, of $\gamma$-H\"{o}lder 
continuous functions on $[0,1]$ with values in $E$ that vanish at zero and at infinity if one uses 
the Schauder basis and coefficient functionals of $C^{[\gamma]}_{0,0}([0,1])$
from \cite[Theorem 2, p.\ 220]{Ciesieski1960} and \cite[Theorem 3, p.\ 230]{Ciesieski1959}.
The results are a bit weaker in both cases since \cite[2.4 Theorem (2), p.\ 138-139]{B2} and 
\cite[5.9 Example b), p.\ 28]{kruse2017} only guarantee that $S_{\mathcal{C}_{0}(\R)}$ and
$S_{\mathcal{C}^{[\gamma]}_{0,0}([0,1])}$ are isomorphisms if $E$ is quasi-complete.

Now, we turn to to spaces of continuously differentiable functions on an interval $(a,b)$ such that
all derivatives can be continuously extended to the boundary.
For an lcHs $E$ and $k\in\N$ 
the space $\mathcal{C}^{k}([a,b],E)$ is given by 
\begin{align*}
 \mathcal{C}^{k}([a,b],E):=\{f\in\mathcal{C}^{k}((a,b),E)\;|\;&(\partial^{\beta})^{E}f\;
 \text{continuously extendable on}\;[a,b]\\
 &\text{for all}\;\beta\in\N_{0},\,\beta\leq k\}
\end{align*}
and equipped with the system of seminorms given by 
\[
 |f|_{\alpha}:=\sup_{\substack{x\in \Omega\\ \beta\in\N_{0}, \beta\leq k}}p_{\alpha}\bigl((\partial^{\beta})^{E}f(x)\bigr),
 \quad f\in\mathcal{C}^{k}([a,b],E),
\]
for $\alpha\in \mathfrak{A}$.
From the Schauder hat functions $(\varphi^{\mathcal{T}}_{n})$ for a dense sequence 
$\mathcal{T}:=(t_{n})_{n\in\N_{0}}$ in $[a,b]$ with $t_{0}=a$, $t_{1}=b$ and 
$t_{n}\neq t_{m}$ for $n\neq m$ and the associated coefficient functionals 
$\lambda_{n}^{\mathbb{K}}$ we can easily get 
a Schauder basis for the space $\mathcal{C}^{k}([a,b])$, $k\in\N$, by 
applying $\int_{a}^{(\cdot)}$ $k$-times to the series representation 
\[
f^{(k)}=\sum_{n=0}^{\infty}\lambda_{n}^{\K}(f^{(k)})\varphi^{\mathcal{T}}_{n},
\quad f\in\mathcal{C}^{k}([a,b]),
\]
where we identified $f^{(k)}$ with its continuous extension on $[a,b]$. 
The resulting Schauder basis $f_{n}^{\mathcal{T}}\colon [a,b]\to \R$ and associated coefficient 
functionals $\mu_{n}^{\K}\colon \mathcal{C}^{k}([a,b])\to \K$, $n\in\N_{0}$, are 
\begin{align*}
f_{n}^{\mathcal{T}}(x)&=\frac{1}{n!}(x-a)^{n}, &&\mu_{n}^{\K}(f)=f^{(n)}(a), & 0\leq n\leq k-1,\\
f_{n}^{\mathcal{T}}(x)&=\int_{a}^{x}\int_{a}^{t_{k-1}}\cdots\int_{a}^{t_{2}}\int_{a}^{t_{1}}
\varphi^{\mathcal{T}}_{n-k}\d t\d t_{1}\dots \d t_{k-1}, 
&&\mu_{n}^{\K}(f)=\lambda_{n-k}^{\K}(f^{(k)}), & n\geq k,
\end{align*}
for $x\in[a,b]$ and $f\in\mathcal{C}^{k}([a,b])$ (see e.g.\ \cite[p.\ 586-587]{schonefeld1969}, 
\cite[2.3.7, p.\ 29]{Semadeni1982}). 
Again, the mapping rule for the coefficient functionals still makes sense if $f\in\mathcal{C}^{k}([a,b],E)$ 
and so we define $\mu^{E}_{n}$ on $\mathcal{C}^{k}([a,b],E)$ for $n\in\N_{0}$ accordingly.

\begin{thm}\label{thm:cont.diff.func.interval}
 Let $E$ be an lcHs with metric ccp, $k\in\N$, $\mathcal{T}:=(t_{n})_{n\in\N_{0}}$ a dense sequence in $[a,b]$ 
 with $t_{0}=a$, $t_{1}=b$ and $t_{n}\neq t_{m}$ for $n\neq m$. 
 Then $(\sum_{n=0}^{l}\mu_{n}^{E}f^{\mathcal{T}}_{n})_{l\in\N_{0}}$ is a Schauder decomposition of 
 $\mathcal{C}^{k}([a,b],E)$ and 
 \[
 f=\sum_{n=0}^{\infty}\mu^{E}_{n}(f)f^{\mathcal{T}}_{n},\quad f\in \mathcal{C}^{k}([a,b],E).
 \]
\end{thm}
\begin{proof}
The spaces $\mathcal{C}^{k}([a,b])$ and $\mathcal{C}^{k}([a,b],E)$ are $\varepsilon$-compatible by 
\cite[5.11 Example, p.\ 29]{kruse2017} if $E$ has metric ccp. 
The Banach space $\mathcal{C}^{k}([a,b])$ is barrelled giving the equicontinuity of its Schauder basis. 
Due to \cite[4.12 Proposition, p.\ 22]{kruse2017} we have for all $u\in\mathcal{C}^{k}([a,b])\varepsilon E$, 
$\beta\in\N_{0}$, $\beta\leq k$, and $x\in(a,b)$
\[
(\partial^{\beta})^{E}S(u)(x)=u(\delta_{x}\circ(\partial^{\beta})^{\K}).
\]
Further, for every sequence $(x_{n})$ in $(a,b)$ converging to $t\in\{a,b\}$ we obtain by 
\cite[4.8 Proposition, p.\ 22]{kruse2017} in combination with \cite[4.9 Lemma, p.\ 22]{kruse2017} 
applied to $T:=(\partial^{\beta})^{\K}$
\[
\lim_{n\to\infty}(\partial^{\beta})^{E}S(u)(x_{n})=u(\lim_{n\to\infty}\delta_{x_{n}}\circ(\partial^{\beta})^{\K}).
\]
From these observations we deduce that $\mu_{n}^{E}(S(u))=u(\mu_{n}^{\K})$ for all $n\in\N_{0}$, i.e.\ 
\eqref{eq:schauder_decomp} holds. Therefore our statement is a consequence of \prettyref{cor:schauder_decomp}.
\end{proof}

\section*{Holomorphic functions}

In this short subsection we show how to get the result on power series expansion of holomorphic functions from the introduction. 
Let $E$ be an lcHs over $\C$, $z_{0}\in\C$, $r\in(0,\infty]$ and $\Omega:=\D_{r}(z_{0})$ 
where $\D_{r}(z_{0})\subset\C$ is an open disc around $z_{0}$ with radius $r>0$. We equip $\mathcal{O}(\Omega,E)$ 
with the system of seminorms given by
\[
 |f|_{K,\alpha}:=\sup_{z\in K}p_{\alpha}(f(z)),\quad f\in \mathcal{O}(\Omega,E),
\]
for $K\subset\Omega$ compact and $\alpha\in\mathfrak{A}$. We note that 
\[
\mathcal{O}(\Omega,E)=\{f\in\mathcal{C}^{\infty}(\Omega,E)\;|\;\overline{\partial}^{E}f=0\}
\]
if $E$ is locally complete where 
\[
\overline{\partial}^{E}f(z)
:=\frac{1}{2}((\partial^{e_{1}})^{E}+i(\partial^{e_{2}})^{E})f(z),
\quad f\in\mathcal{C}^{1}(\Omega,E),\; z\in\Omega,
\]
is the Cauchy-Riemann operator. Moreover, we set $(\partial^{0}_{\C})^{E}f:=f$, 
\[
 (\partial^{1}_{\C})^{E}f(z):=\lim_{\substack{h\to 0\\ h\in\C,h\neq 0}}\frac{f(z+h)-f(z)}{h},\quad z\in \Omega,
\]
and $(\partial^{n+1}_{\C})^{E}f:=(\partial^{1}_{\C})^{E}((\partial^{n}_{\C})^{E}f)$ for $n\in\N_{0}$ 
and $f\in \mathcal{O}(\Omega,E)$ (see \eqref{intro:holom}). We observe that
real and complex derivatives are related by
\begin{equation}\label{eq:complex.real.deriv}
 (\partial^{\beta})^{E}f(z)=i^{\beta_{2}}(\partial^{|\beta|}_{\C})^{E}f(z),\quad z\in\Omega,
\end{equation}
for every $f\in\mathcal{O}(\Omega,E)$ and $\beta=(\beta_{1},\beta_{2})\in\N_{0}^{2}$.

\begin{thm}\label{thm:powerseries}
  Let $E$ be a locally complete lcHs over $\C$, $z_{0}\in\C$ and $r\in(0,\infty]$. 
  Then $(\sum_{n=0}^{k}\frac{(\partial^{n}_{\C})^{E}f(z_{0})}{n!}(\cdot - z_{0})^{n})_{k\in\N_{0}}$ is a Schauder decomposition 
  of $\mathcal{O}(\D_{r}(z_{0}),E)$ and 
   \[
    f=\sum_{n=0}^{\infty}\frac{(\partial^{n}_{\C})^{E}f(z_{0})}{n!}(\cdot - z_{0})^{n},\quad f\in\mathcal{O}(\D_{r}(z_{0}),E).
   \]
\end{thm}
\begin{proof}
The spaces $\mathcal{O}(\D_{r}(z_{0}))$ and $\mathcal{O}(\D_{r}(z_{0}),E)$ are $\varepsilon$-compatible by \cite[Theorem 9, p.\ 232]{B/F/J} 
if $E$ is locally complete. Further, the Schauder basis $((\cdot - z_{0})^{n})$ of $\mathcal{O}(\D_{r}(z_{0}))$ is equicontinuous 
since the Fr\'{e}chet space $\mathcal{O}(\D_{r}(z_{0}))$ is barrelled. Due to \cite[4.12 Proposition, p.\ 22]{kruse2017} and 
\eqref{eq:complex.real.deriv} we have for all $u\in\mathcal{O}(\D_{r}(z_{0}))\varepsilon E$ 
and $z\in\D_{r}(z_{0})$
\[
(\partial^{n}_{\C})^{E}S(u)(z)=u(\delta_{z}\circ(\partial^{n}_{\C})^{\C})
\]
which means that \eqref{eq:schauder_decomp} is satisfied. \prettyref{cor:schauder_decomp} implies our statement.
\end{proof}

\section*{Fourier expansions}

In this subsection we turn our attention to Fourier expansions in the space of vector-valued 
rapidely decreasing functions and in the space of vector-valued smooth, $2\pi$-periodic functions. 
We start with the definition of the Pettis-integral 
which we need to define the Fourier coefficients for vector-valued functions.
For a measure space $(X, \Sigma, \mu)$ and $1\leq p<\infty$ let 
\[
\mathfrak{L}^{p}(X,\mu):=\{f\colon X\to\K\;\text{measurable}\;|\;
q_{p}(f):=\int_{X}|f(x)|^{p}\d\mu(x)<\infty\}
\]
and define the quotient space of $p$-integrable functions by 
$\mathcal{L}^{p}(X,\mu):=\mathfrak{L}^{p}(X,\mu)/\{f\in\mathfrak{L}^{p}(X,\mu)\;|\;q_{p}(f)=0\}$
which becomes a Banach space if it is equipped with the norm 
$\|f\|_{p}:=q_{p}(F)^{1/p}$, $f=[F]\in \mathcal{L}^{p}(X,\mu)$. 
From now on we do not distinguish between equivalence classes and 
their representants anymore.

For a measure space $(X, \Sigma, \mu)$ and $f\colon X\to \K$ we say that 
$f$ is integrable on $\Lambda\in\Sigma$ and write $f\in\mathcal{L}^{1}(\Lambda,\mu)$ 
if $\chi_{\Lambda}f\in \mathcal{L}^{1}(X,\mu)$. Then we set 
\[
\int\limits_{\Lambda} f(x)  \d\mu(x):=\int\limits_{X} \chi_{\Lambda}(x)f(x)  \d\mu(x).
\]

\begin{defn}[Pettis-integral]\label{def:integral}
 	Let $(X, \Sigma, \mu)$ be a measure space and $E$ an lcHs. 
 	A function $f\colon X \to E$ is called weakly (scalarly) measurable if the function 
 	$e'\circ  f  \colon X\to \K$, $ (e'\circ f )(x) := \langle e' , f(x) \rangle:=e'(f(x)),$ 
 	is measurable for all $e' \in E'$.  
 	A weakly measurable function is said to be weakly (scalarly) integrable 
 	if $e' \circ f \in \mathcal{L}^{1}(X,\mu)$. 
 	A function $f\colon X\to E$ is called Pettis-integrable on $\Lambda\in\Sigma$ 
 	if it is weakly integrable on $\Lambda$ and
 	\[
 	\exists\; e_{\Lambda} \in E \; \forall e' \in E': \langle e' , e_{\Lambda} \rangle 
 	= \int\limits_{\Lambda} \langle e' , f(x) \rangle \d\mu(x). 
 	\]  
 	In this case $e_{\Lambda}$ is unique due to $E$ being Hausdorff and we set 
 	\[
 	 \int\limits_{\Lambda} f(x) \d\mu(x):=e_{\Lambda}.
 	\]	
\end{defn}	 

If we consider the measure space $(X, \mathscr{L}(X), \lambda)$ of Lebesgue measurable sets 
for $X\subset\R^{d}$, we just write $\d x:=\d\lambda(x)$.

\begin{lem}\label{lem:pettis.loc.complete}
 Let $E$ be a locally complete lcHs, $\Omega\subset\R^{d}$ open and $f\colon\Omega\to E$.
 If $f$ is weakly $\mathcal{C}^{1}$, i.e.\ $e'\circ f\in\mathcal{C}^{1}(\Omega)$ 
 for every $e'\in E'$, then $f$ is Pettis-integrable on every compact subset 
 of $K\subset\Omega$ with respect to any locally finite measure $\mu$ on $\Omega$ and
 \[
  p_{\alpha}\bigl(\int_{K}f(x)\d\mu(x)\bigr)\leq \mu(K)\sup_{x\in K}p_{\alpha}(f(x)),\quad\alpha\in\mathfrak{A}.
 \]
\end{lem}
\begin{proof}
 Let $K\subset\Omega$ be compact and $(\Omega,\Sigma,\mu)$ a measure space with locally finite measure $\mu$, 
 i.e.\ $\Sigma$ contains the Borel $\sigma$-algebra $\mathcal{B}(\Omega)$ on $\Omega$ and 
 for every $x\in\Omega$ there is a neighbourhood $U_{x}\subset\Omega$ of $x$ such that $\mu(U_{x})<\infty$. 
 Since the map $e'\circ f$ is differentiable for every $e'\in E'$, thus Borel-measurable, 
 and $\mathcal{B}(\Omega)\subset \Sigma$, it is measurable. 
 We deduce that $e'\circ f\in \mathcal{L}^{1}(K,\mu)$ for every $e'\in E'$ because locally finite measures 
 are finite on compact sets. Hence the map 
 \[
  I\colon E'\to \mathbb{K},\;I(e'):=\int_{K}\langle e',f(x)\rangle\d\mu(x)
 \]
 is well-defined and linear. We estimate
 \[
  |I(e')|\leq \mu(K) \sup_{x\in f(K)}|e'(x)|\leq\mu(K) \sup_{x\in \oacx(f(K))}|e'(x)|,\quad e'\in E'.
 \]
 Due to $f$ being weakly $\mathcal{C}^{1}$ and \cite[Proposition 2, p.\ 354]{Bonet2002} 
 the absolutely convex set $\oacx(f(K))$ is compact 
 yielding $I\in(E_{\kappa}')'\cong E$ by the theorem of Mackey-Arens which means that there is $e_{K}\in E$ such that
 \[
  \langle e',e_{K}\rangle=I(e')=\int_{K}\langle e',f(x)\rangle\d\mu(x),\quad e'\in E'.
 \]
 Therefore $f$ is Pettis-integrable on $K$ with respect to $\mu$.
 For $\alpha\in\mathfrak{A}$ we set $B_{\alpha}:=\{x\in E\;|\;p_{\alpha}(x)<1\}$ and observe that
\begin{align*}
 p_{\alpha}\bigl(\int_{K}f(x)\d\mu(x)\bigr)
 &=\sup_{e'\in B_{\alpha}^{\circ}}\bigl|\langle e',\int_{K}f(x)\d\mu(x)\rangle\bigr|
 =\sup_{e'\in B_{\alpha}^{\circ}}\bigl|\int_{K}e'(f(x))\d\mu(x)\bigr|\\
 &\leq \mu(K)\sup_{e'\in B_{\alpha}^{\circ}}\sup_{x\in K}|e'(f(x))|
 = \mu(K)\sup_{x\in K}p_{\alpha}(f(x))
\end{align*}
where we used \cite[Proposition 22.14, p.\ 256]{meisevogt1997} in the first and last equation 
to get from $p_{\alpha}$ to $\sup_{e'\in B_{\alpha}^{\circ}}$ and back.
\end{proof}

For an lcHs $E$ we define the Schwartz space of $E$-valued rapidely decreasing functions by
\[
\mathcal{S}(\R^{d},E)
:=\{f\in \mathcal{C}^{\infty}(\R^{d},E)\;|\;
\forall\;l\in\N,\,\alpha\in \mathfrak{A}:\;|f|_{l,\alpha}<\infty\}
\]
where 
\[
 |f|_{l,\alpha}:=\sup_{\substack{x\in\R^{d}\\ \beta\in\N_{0}^{d}, |\beta|\leq l}}
 p_{\alpha}\bigl((\partial^{\beta})^{E}f(x)\bigr)(1+|x|^{2})^{l/2}.
\]
We recall the definition of the Hermite functions. For $n\in\N_{0}$ we set 
\[
h_{n}\colon\R\to\R,\;
h_{n}(x):=(2^{n}n!\sqrt{\pi})^{-1/2}(x-\frac{d}{dx})^{n}e^{-x^{2}/2}=(2^{n}n!\sqrt{\pi})^{-1/2}H_{n}(x)e^{-x^{2}/2},
\]
with the Hermite polynomials $H_{n}$ of degree $n$ which can be computed recursively by 
\[
 H_{0}(x)=1\quad\text{and}\quad H_{n+1}(x)=2xH_{n}(x)-H_{n}'(x), \quad x\in\R,\; n\in\N_{0}.
\]
For $n=(n_{k})\in\N_{0}^{d}$ we define the $n$-th Hermite function by
\[
h_{n}\colon\R^{d}\to\R,\;
h_{n}(x):=\prod_{k=1}^{d} h_{n_{k}}(x_{k}),\quad\text{and}\quad
H_{n}\colon\R^{d}\to\R,\;
H_{n}(x):=\prod_{k=1}^{d} H_{n_{k}}(x_{k}).
\]

\begin{prop}\label{prop:Schwartz_pettis_int}
Let $E$ be a sequentially complete lcHs, $f\in\mathcal{S}(\R^{d},E)$ and $n\in\N_{0}^{d}$. 
Then $fh_{n}$ is Pettis-integrable on $\R^{d}$.
\end{prop}
\begin{proof}
For $k\in\N$ we define the Pettis-integral
\[
 e_{k}:=\int_{[-k,k]^{d}}f(x)h_{n}(x)\d x
\]
which is a well-defined element of $E$ by Lemma \ref{lem:pettis.loc.complete}. We claim that $(e_{k})$ 
is a Cauchy sequence in $E$. First, we notice that there are $j\in\N$ and $C>0$ such 
that $|H_{n}(x)|\leq C(1+|x|^{2})^{j/2}$ for all $x\in\R^{d}$ as $H_{n}$ is a product of polynomials 
in one variable. Now, let $\alpha\in\mathfrak{A}$,  $k,m\in\N$, $k>m$, and 
set $C_{n}:=(\prod_{i=1}^{n}2^{n_{i}}n_{i}!\sqrt{\pi})^{-1/2}$ as well as $Q_{k,m}:=[-k,k]^{d}\setminus [-m,m]^{d}$. 
We observe that $|\prod_{i=1}^{d}x_{i}|\geq 1$ for every $x\in Q_{k,m}$ and 
\begin{flalign*}
&\quad p_{\alpha}(e_{k}-e_{m})\\
&=\sup_{e'\in B_{\alpha}^{\circ}}|e'(e_{k}-e_{m})|
=\sup_{e'\in B_{\alpha}^{\circ}}\bigl| \int_{Q_{k,m}}\langle e',f(x)h_{n}(x)\rangle\d x\bigr|\\
&\leq C_{n}\int_{Q_{k,m}}e^{-|x|^{2}/2}\d x 
\sup_{e'\in B_{\alpha}^{\circ}}\sup_{x\in \R^{d}}|e'(f(x)H_{n}(x))|\\
&= C_{n}\int_{Q_{k,m}}e^{-|x|^{2}/2}\d x 
\sup_{x\in \R^{d}}p_{\alpha}(f(x))|H_{n}(x)|\\
&\leq C_{n}C|f|_{j,\alpha}\int_{Q_{k,m}}\bigl|\prod_{i=1}^{d}x_{i}\bigr|e^{-|x|^{2}/2}\d x \\
&=C_{n}C|f|_{j,\alpha}\bigl((2\int_{[0,k]}xe^{-x^{2}/2}\d x)^{d}-(2\int_{[0,m]}xe^{-x^{2}/2}\d x)^{d}\bigr)\\
&=2^{d}C_{n}C|f|_{j,\alpha}\bigl((1-e^{-k^{2}/2})^{d}-(1-e^{-m^{2}/2})^{d}\bigr)
\end{flalign*}
proving our claim. Since $E$ is sequentially complete, the limit $y:=\lim_{k\to\infty} e_{k}$ exists in $E$. 
From $e'\circ f\in\mathcal{S}(\R^{d})$ for every $e'\in E'$ and the dominated convergence theorem we deduce
\[
 \langle e',y\rangle=\lim_{k\to\infty} \langle e',e_{k}\rangle=\lim_{k\to\infty}\int_{[-k,k]^{d}}\langle e',f(x)\rangle h_{n}(x)\d x
 =\int_{\R^{d}}\langle e',f(x)h_{n}(x)\rangle \d x, \quad e'\in E',
\]
which yields the Pettis-integrability of $fh_{n}$ on $\R^{d}$ with $\int_{\R^{d}}f(x)h_{n}(x)\d x= y$. 
\end{proof}

Due to the previous proposition we can define the $n$-th Fourier coefficient of 
$f\in\mathcal{S}(\R^{d},E)$ by
\[
\widehat{f}(n):=\int_{\R^{d}}f(x)\overline{h_{n}(x)}\d x=\int_{\R^{d}}f(x)h_{n}(x)\d x, \quad n\in\N_{0}^{d},
\]
if $E$ is sequentially complete.
We know that the map 
\[
 \mathscr{F}^{\K}\colon\mathcal{S}(\R^{d})\to s(\N_{0}^{d}),\;
 \mathscr{F}^{\K}(f):=\bigl(\widehat{f}(n)\bigr)_{n\in\N_{0}^{d}},
\]
is an isomorphism and its inverse is given by 
\[
 (\mathscr{F}^{\K})^{-1}\colon s(\N_{0}^{d})\to \mathcal{S}(\R^{d}),\;
 (\mathscr{F}^{\K})^{-1}(x):=\sum_{n\in\N_{0}^{d}}x_{n}h_{n},
\]
(see e.g.\ \cite[Satz 3.7, p.\ 66]{Kaballo}). 
It is already known that $\mathcal{S}(\R^{d})$ and $\mathcal{S}(\R^{d},E)$ are $\varepsilon$-compatible 
if $E$ is quasi-complete by \cite[Proposition 9, p.\ 108, Th\'{e}or\`{e}me 1, p.\ 111]{Schwartz1955} 
(cf.\ \cite[3.22 Example, p.\ 14]{kruse2017}). We improve this to sequentially complete $E$ 
and derive a Schauder decomposition of $\mathcal{S}(\R^{d},E)$ as well using 
our observations on the sequence space $s(\N_{0}^{d},E)$. 

\begin{thm}\label{thm:fourier.rap.dec}
Let $E$ be a sequentially complete lcHs. Then the following holds.
\begin{enumerate}
 \item [a)] If $E$ is sequentially complete, then 
 \[
  \mathscr{F}^{E}\colon \mathcal{S}(\R^{d},E)\to s(\N_{0}^{d},E),
  \;\mathscr{F}^{E}(f):=\bigl(\widehat{f}(n)\bigr)_{n\in\N_{0}^{d}},
 \]
 is an isomorphism, $\mathcal{S}(\R^{d},E)\cong\mathcal{S}(\R^{d})\varepsilon E$ and
 \[
 \mathscr{F}^{E}=S_{s(\N_{0}^{d})}\circ(\mathscr{F}^{\K}\varepsilon\id_{E})\circ S_{\mathcal{S}(\R^{d})}^{-1}.
 \]
\item [b)] $(\sum_{n\in\N_{0}^{d}, |n|\leq k}\mathscr{F}^{E}_{n}h_{n})_{k\in\N}$ 
is a Schauder decomposition of $\mathcal{S}(\R^{d},E)$ and 
\[
f=\sum_{n\in\N_{0}^{d}}\widehat{f}(n)h_{n}, \quad f\in\mathcal{S}(\R^{d},E).
\]
\end{enumerate}
\end{thm}
\begin{proof}
First, we show that the map $\mathscr{F}^{E}$ is well-defined. Let $f\in\mathcal{S}(\R^{d},E)$. 
Then $e'\circ f\in \mathcal{S}(\R^{d})$ and 
\[
 \langle e',\mathscr{F}^{E}(f)_{n}\rangle=\langle e',\widehat{f}(n)\rangle
 =\widehat{e'\circ f}(n)=\mathscr{F}^{\K}(e'\circ f)_{n}
\]
for every $n\in\N_{0}^{d}$ and $e'\in E'$. Thus we have $\mathscr{F}^{\K}(e'\circ f)\in s(\N_{0}^{d})$ 
for every $e'\in E'$ which implies by \cite[Mackey's theorem 23.15, p.\ 268]{meisevogt1997}
that $\mathscr{F}^{E}(f)\in s(\N_{0}^{d},E)$ and that $\mathscr{F}^{E}$ is well-defined.

We notice that
\[
 s(\N_{0}^{d},E) \cong s(\N_{0}^{d})\varepsilon E
 \cong \mathcal{S}(\mathbb{R}^{d})\varepsilon E 
 \tilde{\hookrightarrow} \mathcal{S}(\mathbb{R}^{d},E)
\]
where the first isomorphism is $S^{-1}_{s(\N_{0}^{d})}$ by \prettyref{thm:sequence.spaces} (i), 
the second is the map $(\mathscr{F}^{\K})^{-1}\varepsilon\id_{E}$
and the third isomorphism into is the map $S_{\mathcal{S}(\R^{d})}$ by 
\cite[3.9 Theorem, p.\ 9]{kruse2017} and \cite[4.12 Proposition, p.\ 22]{kruse2017} 
since $\mathcal{S}(\R^{d})$ is barrelled. 
Next, we show that $S_{\mathcal{S}(\R^{d})}\circ \bigl((\mathscr{F}^{\K})^{-1}\varepsilon\id_{E}\bigr) 
\circ S_{s(\N_{0}^{d})}^{-1}$ is surjective and the inverse of 
$\mathscr{F}^{E}$. We can explicitely compute the composition of these maps.
By the proof of \prettyref{prop:reverse_Schauder} b) we get
that the inverse of $S_{s(\N_{0}^{d})}$ is given by 
\[
 S_{s(\N_{0}^{d})}^{-1}\colon s(\N_{0}^{d},E) \to s(\N_{0}^{d}) \varepsilon E,\;
 w\mapsto \lim_{k\to\infty}\sum_{n\in\N_{0}^{d},|n|\leq k}\Theta(\varphi_{n}\otimes w_{n}).
\]
Let $w\in s(\N_{0}^{d},E)$. Then for every $x\in\R^{d}$ and $e'\in E'$
\begin{align*}
S_{s(\N_{0}^{d})}^{-1}\Bigl(\bigl((\mathscr{F}^{\K})^{-1}\bigr)^{t}(\delta_{x})\Bigr)
&=\sum_{n\in\N_{0}^{d}}\Theta(\varphi_{n}\otimes w_{n})\Bigl(\bigl((\mathscr{F}^{\K})^{-1}\bigr)^{t}(\delta_{x})\Bigr)\\
&=\sum_{n\in\N_{0}^{d}}(\mathscr{F}^{\K})^{-1}(\varphi_{n})(x)w_{n}
 =\sum_{n\in\N_{0}^{d}}w_{n}h_{n}(x)
\end{align*}
which gives 
\[
\bigl(S_{\mathcal{S}(\R^{d})}\circ \bigl((\mathscr{F}^{\K})^{-1}\varepsilon\id_{E}\bigr)
\circ S_{s(\N_{0}^{d})}^{-1}\bigr)(w)(x)
=\sum_{n\in\N_{0}^{d}}w_{n}h_{n}(x).
\]
Let $f\in\mathcal{S}(\R^{d},E)$. Then $w:=\mathscr{F}^{E}(f)\in s(\N_{0}^{d},E)$ and 
\begin{align*}
 e'(\sum_{n\in\N_{0}^{d}}w_{n}h_{n})
 &=\sum_{n\in\N_{0}^{d}}e'(\widehat{f}(n))h_{n}
 =\sum_{n\in\N_{0}^{d}}\widehat{e'\circ f}(n)h_{n}\\
 &=(\mathscr{F}^{\K})^{-1}\bigl((\widehat{e'\circ f}(n))_{n\in\N_{0}^{d}}\bigr)
 = e'\circ f
\end{align*}
for all $e'\in E'$ resulting in 
\[
f=\sum_{n\in\N_{0}^{d}}w_{n}h_{n}=\sum_{n\in\N_{0}^{d}}\mathscr{F}^{E}(f)_{n}h_{n}
\]
and
\[
\bigl(S_{\mathcal{S}(\R^{d})}\circ \bigl((\mathscr{F}^{\K})^{-1}\varepsilon\id_{E}\bigr)
\circ S_{s(\N_{0}^{d})}^{-1}\bigr)(w)(x)=f(x)
\]
for every $x\in\R^{d}$. Thus we conclude 
\[
 \bigl[\bigl(S_{\mathcal{S}(\R^{d})}\circ \bigl((\mathscr{F}^{\K})^{-1}\varepsilon\id_{E}\bigr)
 \circ S_{s(\N_{0}^{d})}^{-1}\bigr)\circ \mathscr{F}^{E}\bigr](f)=f
\]
yielding the surjectivity of the composition. 
Therefore $S_{\mathcal{S}(\R^{d})}\circ \bigl((\mathscr{F}^{\K})^{-1}\varepsilon\id_{E}\bigr)
\circ S_{s(\N_{0}^{d})}^{-1}$ is an isomorphism 
with right inverse $\mathscr{F}^{E}$ which implies that $\mathscr{F}^{E}$ is its inverse. 
In addition, the bijectivity of 
$S_{\mathcal{S}(\R^{d})}\circ \bigl((\mathscr{F}^{\K})^{-1}\varepsilon\id_{E}\bigr)
\circ S_{s(\N_{0}^{d})}^{-1}$ and $S_{s(\N_{0}^{d})}^{-1}$ implies that 
$S_{\mathcal{S}(\R^{d})}\circ \bigl((\mathscr{F}^{\K})^{-1}\varepsilon\id_{E}\bigr)$ 
is bijective and thus $S_{\mathcal{S}(\R^{d})}$ is surjective. 
Hence $S_{\mathcal{S}(\R^{d})}$ is also an isomorphism completing the proof of part a).

The rest of part b) follows from the isomorphy $\mathcal{S}(\R^{d},E)\cong s(\N_{0}^{d},E)$ 
via $\mathscr{F}^{E}$ and \prettyref{prop:vector_valued_seq_spaces} a).

\end{proof}

Our last example of this subsection is devoted to Fourier expansions of vector-valued $2\pi$-periodic smooth functions. 
We equip the space $\mathcal{C}^{\infty}(\R^{d},E)$ for locally convex Hausdorff $E$ 
with the system of seminorms generated by 
\[
|f|_{K,l,\alpha}:=\hspace{-0.2cm}\sup_{\substack{x\in \Omega\\ \beta\in\N_{0}^{d},|\beta|\leq l}}
p_{\alpha}((\partial^{\beta})^{E}f(x))\chi_{K}(x)
=\hspace{-0.2cm}\sup_{\substack{x\in K\\ \beta\in\N_{0}^{d},|\beta|\leq l}}
p_{\alpha}((\partial^{\beta})^{E}f(x)),\quad f\in\mathcal{C}^{\infty}(\R^{d},E),
\]
for $K\subset\R^{d}$ compact, $l\in\N_{0}$ and $\alpha\in\mathfrak{A}$.
By $\mathcal{C}^{\infty}_{2\pi}(\R^{d},E)$ we denote the topological subspace of 
$\mathcal{C}^{\infty}(\R^{d},E)$ consisting of the functions 
which are $2\pi$-periodic in each variable. 
Due to \prettyref{lem:pettis.loc.complete} we are able to define the $n$-th Fourier coefficient of 
$f\in\mathcal{C}^{\infty}_{2\pi}(\R^{d},E)$ by
\[
\widehat{f}(n):=(2\pi)^{-d}\int_{[-\pi,\pi]^{d}}f(x)e^{-i\langle n,x\rangle}\d x, \quad n\in\Z^{d},
\]
where $\langle\cdot,\cdot\rangle$ is the usual scalar product on $\R^{d}$, if $E$ is locally complete. 

\begin{thm}\label{thm:fourier}
Let $E$ be a locally complete lcHs over $\C$. 
Then a Schauder decomposition of $\mathcal{C}^{\infty}_{2\pi}(\R^{d},E)$ is given by 
$(\sum_{n\in\Z^{d}, |n|\leq k}\widehat{f}(n)e^{i\langle n,\cdot\rangle})_{k\in\N}$ and 
\[
f=\sum_{n\in\Z^{d}}\widehat{f}(n)e^{i\langle n,\cdot\rangle}, \quad f\in\mathcal{C}^{\infty}_{2\pi}(\R^{d},E).
\]
\end{thm}
\begin{proof}
First, we prove that $\mathcal{C}^{\infty}_{2\pi}(\R^{d})$ and $\mathcal{C}^{\infty}_{2\pi}(\R^{d},E)$ 
are $\varepsilon$-compatible. It follows from \cite[p.\ 228]{B/F/J} (cf.\ \cite[3.23 Example, p.\ 15]{kruse2017}) 
that $S_{\mathcal{C}^{\infty}(\R^{d})}\colon \mathcal{C}^{\infty}(\R^{d})\varepsilon E\to \mathcal{C}^{\infty}(\R^{d},E)$ 
is an isomorphism. 
Furthermore, for each $x\in\R^{d}$ and $1\leq n\leq d$ we have $\delta_{x}=\delta_{x+2\pi e_{n}}$ in $\mathcal{C}^{\infty}_{2\pi}(\R^{d})'$ and thus
\[
 S_{\mathcal{C}^{\infty}_{2\pi}(\R^{d})}(u)(x)-S_{\mathcal{C}^{\infty}_{2\pi}(\R^{d})}(u)(x+2\pi e_{n})
 =u(\delta_{x}-\delta_{x+2\pi e_{n}})=0,\quad u\in\mathcal{C}^{\infty}_{2\pi}(\R^{d})\varepsilon E,
\]
implying that $S_{\mathcal{C}^{\infty}_{2\pi}(\R^{d})}(u)$ is $2\pi$-periodic in each variable. 
In addition, we observe that $e'\circ f$ is $2\pi$-periodic in each 
variable for all $e'\in E'$ and $f\in\mathcal{C}^{\infty}_{2\pi}(\R^{d},E)$. 
An application of \cite[3.26 Proposition (i), p.\ 17]{kruse2017} yields that 
$S_{\mathcal{C}^{\infty}_{2\pi}(\R^{d})}\colon \mathcal{C}^{\infty}_{2\pi}(\R^{d})\varepsilon E
\to \mathcal{C}^{\infty}_{2\pi}(\R^{d},E)$ is an isomorphism, i.e.\ 
$\mathcal{C}^{\infty}_{2\pi}(\R^{d})$ and $\mathcal{C}^{\infty}_{2\pi}(\R^{d},E)$ 
are $\varepsilon$-compatible.

The space $\mathcal{C}^{\infty}_{2\pi}(\R^{d})$ is barrelled since it is a Fr\'{e}chet space 
and thus its Schauder basis $(e^{i\langle n,\cdot\rangle})$ is equicontinuous. 
By \cite[3.17 Theorem, p.\ 12]{kruse2017} the inverse of $S_{\mathcal{C}^{\infty}_{2\pi}(\R^{d})}$ is given by 
\[
 R^{t}\colon \mathcal{C}^{\infty}_{2\pi}(\R^{d},E) \to \mathcal{C}^{\infty}_{2\pi}(\R^{d}) \varepsilon E,
 \;f\mapsto \mathcal{J}^{-1}\circ R_{f}^{t},
\]
where $\mathcal{J}\colon E\to E'^{\star}$ is the canonical injection in the algebraic dual $E'^{\star}$ of $E'$ and 
\[
 R_{f}^{t}\colon \mathcal{C}^{\infty}_{2\pi}(\R^{d})'\to E'^{\star},\;y\longmapsto \bigl[e'\mapsto y(e'\circ f) \bigr],
\]
for $f\in\mathcal{C}^{\infty}_{2\pi}(\R^{d},E)$.
Let $u\in\mathcal{C}^{\infty}_{2\pi}(\R^{d}) \varepsilon E$
Then $f:=S_{\mathcal{C}^{\infty}_{2\pi}(\R^{d})}(u)\in\mathcal{C}^{\infty}_{2\pi}(\R^{d},E)$ and from 
the Pettis-integrability of $f$ we obtain
\begin{align*}
 R_{f}^{t}(\mathfrak{F}^{\C}_{n})(e')
 &=\mathfrak{F}^{\C}_{n}(e'\circ f)
 =(2\pi)^{-d}\int_{[-\pi,\pi]^{d}}\langle e',f(x)e^{-i\langle n,x\rangle}\rangle\d x\\
 &=\langle e',(2\pi)^{-d}\int_{[-\pi,\pi]^{d}}f(x)e^{-i\langle n,x\rangle}\d x \rangle
 =\langle e', \mathfrak{F}^{E}_{n}(f)\rangle,\quad e'\in E',
\end{align*}
for all $n\in\Z^{d}$ which results in
\[ 
u(\mathfrak{F}^{\C}_{n})=S_{\mathcal{C}^{\infty}_{2\pi}(\R^{d})}^{-1}(f)(\mathfrak{F}^{\C}_{n})
=\mathcal{J}^{-1}\bigl(R_{f}^{t}(\mathfrak{F}^{\C}_{n})\bigr)
=\mathfrak{F}^{E}_{n}(f)=\mathfrak{F}^{E}_{n}(S_{\mathcal{C}^{\infty}_{2\pi}(\R^{d})}(u))
\]
and thus shows the validity of \eqref{eq:schauder_decomp}. 
Now, our statement follows from \prettyref{cor:schauder_decomp}.
\end{proof}

Considering the coefficients in the series expansion above, we know that the map 
\[
 \mathfrak{F}^{\C}\colon\mathcal{C}^{\infty}_{2\pi}(\R^{d})\to s(\Z^{d}),\;
 \mathfrak{F}^{\C}(f):=\bigl(\widehat{f}(n)\bigr)_{n\in\Z^{d}},
\]
is an isomorphism (see e.g.\ \cite[Satz 1.7, p.\ 18]{Kaballo}). Thus we have the following relation 
if $E$ is a locally complete Hausdorff space over $\C$ 
\[
 \mathcal{C}^{\infty}_{2\pi}(\R^{d},E) \cong \mathcal{C}^{\infty}_{2\pi}(\R^{d})\varepsilon E
 \cong  s(\Z^{d})\varepsilon E \tilde{\hookrightarrow} s(\Z^{d},E)
\]
where the first isomorphism is $S^{-1}_{\mathcal{C}^{\infty}_{2\pi}(\R^{d})}$, 
the second is the map $\mathfrak{F}^{\C}\varepsilon\id_{E}$ 
and the third isomorphism into is the map $S_{s(\Z^{d})}$ by \cite[3.9 Theorem, p.\ 9]{kruse2017}. 
We can explicitely compute the composition of these maps. 
With the notation from the proof above we have for every $f\in\mathcal{C}^{\infty}_{2\pi}(\R^{d},E)$ 
and $n\in\Z^{d}$
\begin{flalign*}
 &\quad\bigl(S_{s(\Z^{d})}\circ (\mathfrak{F}^{\C}\varepsilon\id_{E})
 \circ S^{-1}_{\mathcal{C}^{\infty}_{2\pi}(\R^{d})}\bigr)(f)(n)\\
 &=S_{s(\Z^{d})}\Bigl(\bigl(\mathfrak{F}^{\C}\varepsilon\id_{E}\bigr)
 \bigl(S^{-1}_{\mathcal{C}^{\infty}_{2\pi}(\R^{d})}(f)\bigr)\Bigr)(n)
 =S_{s(\Z^{d})}\bigl(S^{-1}_{\mathcal{C}^{\infty}_{2\pi}(\R^{d})}(f)\circ (\mathfrak{F}^{\C})^{t}\bigr)(n)\\
 &=S^{-1}_{\mathcal{C}^{\infty}_{2\pi}(\R^{d})}(f)\bigl((\mathfrak{F}^{\C})^{t}(\delta_{n})\bigr)
 =S^{-1}_{\mathcal{C}^{\infty}_{2\pi}(\R^{d})}(f)(\mathfrak{F}^{\C}_{n})
 =\mathfrak{F}^{E}_{n}(f)=\widehat{f}(n).
\end{flalign*}
Thus the map 
\[
 \mathfrak{F}^{E}\colon\mathcal{C}^{\infty}_{2\pi}(\R^{d},E)\to s(\Z^{d},E),\;
 \mathfrak{F}^{E}(f):=\bigl(\widehat{f}(n)\bigr)_{n\in\Z^{d}},
\]
is well-defined and an isomorphism into if $E$ is locally complete. If $E$ is sequentially complete, 
it is even an isomorphism to the whole space $s(\Z^{d},E)$ because 
$S_{s(\Z^{d})}$ is surjective to the whole space then 
by \prettyref{thm:sequence.spaces} (i). Hence we have:

\begin{thm}\label{thm:fourier.periodic}
If $E$ is a sequentially complete lcHs over $\C$, 
then $\mathfrak{F}^{E}\colon\mathcal{C}^{\infty}_{2\pi}(\R^{d},E)\to s(\Z^{d},E),\;
\mathfrak{F}^{E}(f):=\bigl(\widehat{f}(n)\bigr)_{n\in\Z^{d}}$, is an isomorphism 
and 
\[
\mathfrak{F}^{E}=S_{s(\Z^{d})}\circ (\mathfrak{F}^{\C}\varepsilon\id_{E})\circ S^{-1}_{\mathcal{C}^{\infty}_{2\pi}(\R^{d})}.
\]
\end{thm}

For quasi-complete $E$ \prettyref{thm:fourier} and \prettyref{thm:fourier.periodic} are already known 
by \cite[Satz 10.8, p.\ 239]{Kaballo}.

\subsection*{Acknowledgement}
The author is deeply grateful to Jos\'e Bonet for many helpful suggestions honing the present paper. 
The main \prettyref{thm:schauder_decomp} is essentially due to him improving a previous version of 
the author which became \prettyref{cor:schauder_decomp}.

\bibliography{biblio}
\bibliographystyle{plain}
\end{document}